\documentclass[a4paper,12pt]{article}
\title{{\bf On a computation of rank two Donaldson-Thomas invariants}}
\date{}
\author{Yukinobu Toda}

\usepackage{makeidx}

\usepackage{latexsym}
\usepackage{amscd}
\usepackage{amsmath}
\usepackage{amssymb}
\usepackage{amsthm}
\usepackage{float}
\usepackage{graphicx}

\usepackage[all,ps,dvips]{xy}

\usepackage{array}
\usepackage{amscd}
\usepackage[all]{xy}
\usepackage{makeidx}
\usepackage{latexsym}
\DeclareFontFamily{U}{rsfs}{%
\skewchar\font127}
\DeclareFontShape{U}{rsfs}{m}{n}{%
<-6>rsfs5<6-8.5>rsfs7<8.5->rsfs10}{}
\DeclareSymbolFont{rsfs}{U}{rsfs}{m}{n}
\DeclareSymbolFontAlphabet
{\mathrsfs}{rsfs}
\DeclareRobustCommand*\rsfs{%
\@fontswitch\relax\mathrsfs}
\theoremstyle{plain}
\newtheorem{thm}{Theorem}[section]
\newtheorem{prop}[thm]{Proposition}
\newtheorem{lem}[thm]{Lemma}

\newtheorem{defi}[thm]{Definition}
\newtheorem{rmk}[thm]{Remark}
\newtheorem{cor}[thm]{Corollary}

\newtheorem{case}{Case}

\newtheorem{prop-defi}[thm]{Proposition-Definition}
\newtheorem{thm-defi}[thm]{Theorem-Definition}
\newtheorem{lem-defi}[thm]{Lemma-Definition}

\newtheorem{exam}[thm]{Example}

\newdimen\argwidth
\def\db[#1\db]{
 \setbox0=\hbox{$#1$}\argwidth=\wd0
 \setbox0=\hbox{$\left[\box0\right]$}
  \advance\argwidth by -\wd0
 \left[\kern.3\argwidth\box0 \kern.3\argwidth\right]}

\newcommand{\aA}{\mathcal{A}}

\newcommand{\dD}{\mathcal{D}}
\newcommand{\eE}{\mathcal{E}}
\newcommand{\fF}{\mathcal{F}}

\newcommand{\hH}{\mathcal{H}}

\newcommand{\mM}{\mathcal{M}}

\newcommand{\oO}{\mathcal{O}}

\newcommand{\qQ}{\mathcal{Q}}

\newcommand{\sS}{\mathcal{S}}

\newcommand{\uU}{\mathcal{U}}

\newcommand{\xX}{\mathcal{X}}
\newcommand{\yY}{\mathcal{Y}}
\newcommand{\zZ}{\mathcal{Z}}
\newcommand{\lr}{\longrightarrow}

\newcommand{\Supp}{\mathop{\rm Supp}\nolimits}
\newcommand{\Hom}{\mathop{\rm Hom}\nolimits}

\newcommand{\Hilb}{\mathop{\rm Hilb}\nolimits}

\newcommand{\id}{\textrm{id}}

\newcommand{\Ext}{\mathop{\rm Ext}\nolimits}
\newcommand{\Spec}{\mathop{\rm Spec}\nolimits}

\newcommand{\Coh}{\mathop{\rm Coh}\nolimits}

\newcommand{\cneq}{\mathrel{\raise.095ex\hbox{:}\mkern-4.2mu=}}
\newcommand{\eqcn}{\mathrel{=\mkern-4.5mu\raise.095ex\hbox{:}}}

\newcommand{\Aut}{\mathop{\rm Aut}\nolimits}

\newcommand{\Stab}{\mathop{\rm Stab}\nolimits}

\newcommand{\Sch}{\mathop{\rm Sch}\nolimits}
\newcommand{\DT}{\mathop{\rm DT}\nolimits}

\newcommand{\groupoid}{\mathop{\rm groupoid}\nolimits}

\newcommand{\Sym}{\mathop{\rm Sym}\nolimits}

\newcommand{\Eu}{\mathop{\rm Eu}\nolimits}

\newcommand{\Quot}{\mathop{\rm Quot}\nolimits}

\newcommand{\Imm}{\mathop{\rm Im}\nolimits}

\newcommand{\GL}{\mathop{\rm GL}\nolimits}
\newcommand{\PGL}{\mathop{\rm PGL}\nolimits}

\newcommand{\tr}{\mathop{\rm tr}\nolimits}
\newcommand{\ex}{\mathop{\rm ex}\nolimits}

\newcommand{\length}{\mathop{\rm length}\nolimits}

\newcommand{\cl}{\mathop{\rm cl}\nolimits}
\begin{document}
\maketitle
\begin{abstract}
For a Calabi-Yau 3-fold $X$, 
we explicitly compute the Donaldson-Thomas type 
invariant counting pairs $(F, V)$, where
$F$ is a zero-dimensional coherent sheaf on $X$ and 
$V\subset F$ is a two dimensional linear subspace, which 
satisfy a certain stability condition. 
This is a rank two version of the DT-invariant of rank one, 
studied by 
 Li, Behrend-Fantechi and 
Levine-Pandharipande. 
We use 
the wall-crossing formula of DT-invariants 
established by Joyce-Song, Kontsevich-Soibelman.  
\end{abstract}
\section{Introduction}
The purpose of this article is to write down a 
closed formula of the generating series of certain rank 
two Donaldson-Thomas (DT) type invariants on 
Calabi-Yau 3-folds. The DT-invariant
is a counting invariant of stable coherent sheaves on $X$, and 
it is 
introduced in~\cite{Thom} in order to 
give a holomorphic analogue of the
Casson invariant on real 3-manifolds.
It is now conjectured by Maulik-Nekrasov-Okounkov-Pandharipande
(MNOP)~\cite{MNOP}
that Gromov-Witten invariants and rank one 
DT-invariants are related in terms of generating functions. 
So far, rank one DT-invariants have been
studied in several papers, e.g.
~\cite{Li}, \cite{BBr}, \cite{LP}, \cite{BeBryan}.
 
On the other hand, it seems that 
higher rank DT-invariants have not been explicitly calculated
yet in any example.
Although the rank one case is important in connection 
with MNOP conjecture, there is also some motivation of 
studying higher rank DT-invariants. For instance, 
the rank of a coherent sheaf is not preserved 
under Fourier-Mukai transformations, e.g. the 
Pfaffian-Grassmannian derived 
equivalence established in~\cite{BorCal}.
Hence in order to compare DT-invariants under Fourier-Mukai 
transformations, it seems that we also have to work with 
higher rank DT-invariants. 

Recently the wall-crossing 
formula of DT-invariants has been developed by Joyce-Song~\cite{JS}
and Kontsevich-Soibelman~\cite{K-S}. 
As pointed out in~\cite[Paragraph~6.5]{K-S}, 
certain higher rank DT-type invariants are in principle 
calculated by the wall-crossing formula, 
if we are given data for the DT-invariants of rank one. 
In this article, we work out the wall-crossing 
formula established by Joyce-Song~\cite{JS}, 
and write down the explicit formula of 
DT-type invariants counting rank two D0-D6 bound state, 
discussed in~\cite[Paragraph~6.5]{K-S}. 
We also give an evidence of the integrality 
conjecture proposed by Kontsevich-Soibelman~\cite[Conjecture~6]{K-S}.

\subsection{Rank one Donaldson-Thomas invariant}
Let $X$ be a smooth projective Calabi-Yau 3-fold 
over $\mathbb{C}$, i.e. $K_X=\wedge^{3} T_{X}^{\ast}$ is trivial
and $H^1(\oO_X)=0$. 
For $n\in \mathbb{Z}$, let $\Hilb^{n}(X)$ is the 
Hilbert scheme of $n$-points in $X$, 
\begin{align*}
\Hilb^{n}(X)&=\{Z\subset X : \dim Z=0, \
\length \oO_Z=n\}, \\ 
&=\left\{(F, v) : \begin{array}{l}
F \mbox{ is a zero-dimensional coherent sheaf on }
X \mbox{ with}\\
\mbox{length }
n, \mbox{ and }
v\in F \mbox{ generates }F \mbox{ as an }\oO_X\mbox{-module}.
\end{array}\right\}.
\end{align*}
The moduli space $\Hilb^{n}(X)$ is projective and has a 
symmetric obstruction theory~\cite{Thom}. 
By integrating the associated zero-dimensional virtual 
cycle, we can define the rank one Donaldson-Thomas (DT) invariant, 
$$\DT(1, n)=\int_{[\Hilb^{n}(X)]^{\rm{vir}}}1 \in \mathbb{Z}.$$
Another way of defining DT-invariant is to use 
Behrend's constructible function~\cite{Beh},
$$\nu \colon \Hilb^{n}(X) \to \mathbb{Z}.$$
In~\cite{Beh}, K.~Behrend shows that 
 $\DT(1, n)$ is also written as 
 $$\DT(1, n)=\int_{\Hilb^{n}(X)}\nu \ d\chi \cneq 
 \sum_{k\in \mathbb{Z}}k \chi(\nu^{-1}(k)),$$
 where $\chi(\ast)$ is the topological Euler 
 characteristic. 
 Let $\DT(1)$ be the generating series, 
 $$\DT(1)=\sum_{n\in \mathbb{Z}}\DT(1, n)q^n.$$
 The series $\DT(1)$ is 
 computed by Li~\cite{Li}, Behrend-Fantechi~\cite{BBr} and
Pandharipande-Levine~\cite{LP}.   
\begin{thm}\emph{\bf{\cite{Li}, \cite{BBr}, \cite{LP}}}
We have the following formula, 
$$\DT(1)=M(-q)^{\chi(X)}.$$
Here $M(q)$ is the MacMahon function, 
$$M(q)=\prod_{k\ge 1}\frac{1}{(1-q)^{k}}.$$
\end{thm}
\subsection{Rank two Donaldson-Thomas invariant}
In this article, we consider a
rank two analogue of 
the invariant $\DT(1, n)$. 
Let $F$ be a zero-dimensional 
coherent sheaf on $X$ with length $n$, 
and $V\subset F$ is a two 
dimensional $\mathbb{C}$-vector subspace. 
We call the pair $(F, V)$ \textit{semistable} 
(resp. \textit{stable}) if it
satisfies the following 
stability condition. 
\begin{itemize}
\item The subspace $V\subset F$ generates $F$ as an 
$\oO_X$-module. 
\item For any non-zero $v\in V$, the subsheaf 
$F_v\cneq \oO_X \cdot v \subset F$ satisfies 
$$\length F_v \ge n/2, \quad (\mbox{resp. }\length F_v > n/2.)$$
\end{itemize}
We denote by $M^{(2, n)}$ the moduli space 
of semistable $(F, V)$ with $\length F=n$. 
If $n$ is odd, the space $M^{(2, n)}$ is an algebraic 
space of finite type, and the integration 
of the Behrend function yields the DT-type invariant, 
\begin{align}\label{intro:DT2n}
\DT(2, n)=\int_{M^{(2, n)}}\nu \ d\chi. 
\end{align}
When $n$ is even, the space $M^{(2, n)}$ is 
an algebraic stack, and the integration 
such as (\ref{intro:DT2n}) does not make sense. 
However we are also able to 
define the DT-type invariant, 
$$\DT(2, n)\in \mathbb{Q}, $$
when $n$ is even by using the 
technique of the Hall-algebra. 
The existence of the above $\mathbb{Q}$-valued 
invariant is one of the big achievement of 
the recent work of Joyce-Song~\cite{JS}. 
We will give a brief introduction of 
the definition of $\DT(2, n)$ in Section~\ref{sec:Hall}.
Let us consider the generating series, 
$$\DT(2)=\sum_{n\in \mathbb{Z}}\DT(2, n)q^n.$$
Applying the wall-crossing formula of DT-invariants~\cite{JS}, \cite{K-S}, 
we show the following formula. 
\begin{thm}
We have the following formula. 
\begin{align}\label{DT2}
\DT(2)=\frac{1}{4}M(q)^{2\chi(X)} -\frac{\chi(X)}{2}\{M(q)^{\chi(X)}
\cdot M(q)^{\chi(X)}\cdot N(q)\}_{\Delta},
\end{align}
where $\Delta \subset \mathbb{Z}^3_{\ge 0}$ is 
$$\Delta=\{(m_1, m_2, m_3)\in \mathbb{Z}_{\ge 0}^3 : 
-m_3 \le m_1-m_2 < m_3\}.$$
\end{thm}
Let us explain the notation. 
The series $N(q)$ is defined by 
\begin{align*}N(q) 
&\cneq q\frac{d}{dq}\log M(q) \\
&=\sum_{r, n\ge 0, \ r|n}r^2 q^n, 
\end{align*}
and for $f_1, f_2, \cdots, f_N \in \mathbb{Q}\db[q\db]$
given by 
$$f_i=\sum_{n\ge 0}a_{n}^{(i)}q^n, \quad 1\le i\le N, $$
and a subset $\Delta \subset \mathbb{Z}^{N}_{\ge 0}$, the series 
$\{f_1 \cdot f_2 \cdots f_N\}_{\Delta}$ is defined by 
$$\{f_1 \cdot f_2 \cdots f_N\}_{\Delta}
=\sum_{(n_1, n_2, \cdots, n_N)\in \Delta}
a_{n_1}^{(1)}a_{n_2}^{(2)}\cdots a_{n_N}^{(N)}q^{n_1+n_2+\cdots +n_N}.$$
In the formula~(\ref{DT2}), we set 
$N=3$, $f_1=f_2=M(q)^{\chi(X)}$ and $f_3=N(q)$. 
\subsection{Integrality property}
Following~\cite{K-S}, we introduce the invariant 
$$\Omega(2, n)=\left\{ \begin{array}{cl}\DT(2, n), & n \mbox{ is odd,} \\
\DT(2, n)-\frac{1}{4}\DT(1, \frac{n}{2}), & n\mbox{ is even.}
\end{array}\right.
$$
We also prove an evidence of the
integrality conjecture by Kontsevich-Soibelman~\cite[Conjecture~6]{K-S}. 
\begin{thm}
We have $\Omega(2, n)\in \mathbb{Z}$ for any $n\in \mathbb{Z}$. 
\end{thm}
A first few terms of $\Omega(2, n)$ are calculated as follows, 
\begin{align*}
\Omega(2, 0)&=\Omega(2, 1)=0, \ \Omega(2, 2)=-\chi, \\
\Omega(2, 3)&=-\frac{1}{6}(\chi^3 +15 \chi^2 +20\chi), \\
\Omega(2, 4)&=-\frac{1}{12}(\chi^4 +30\chi^3 +119\chi^2 +102\chi).
\end{align*}
We note that $\Omega(2, n)$ are numbers which 
fill a part of the marks `?' in~\cite[Paragraph~6.5]{K-S}.
In the very recent paper by Stoppa~\cite{STOP}, 
the invariants have also been computed up to rank three. 
Especially he computed the invariants both 
using Kontsevich-Soibelman formula and Joyce-Song formula. 
He also show the integrarity of Kontsevich-Soibelman's
BPS invariant up to rank three.

\subsection{Acknowledgement}
The author thanks Tom Bridgeland and Dominic Joyce
for valuable discussions. 
He is also grateful to Jacopo Stoppa for informing him 
a mistake in the first version of this paper. 
This work is supported by World Premier International 
Research Center Initiative (WPI initiative), MEXT, Japan. 

\subsection{Notation and convention}
In this paper, all the varieties are defined over 
$\mathbb{C}$. 
For a variety $X$, the abelian category of 
coherent sheaves on $X$ is denoted by $\Coh(X)$. 
The bounded derived category of coherent sheaves on $X$, 
which forms a triangulated category, 
is denoted by $D^b(\Coh(X))$. 
For a triangulated category 
$\dD$, the shift functor is denoted by $[1]$. 
For a set of objects $S\subset \dD$, we denote by 
$\langle S \rangle_{\tr}\subset\dD$ the
smallest triangulated subcategory of $\dD$ which 
contains $S$. Also we denote by $\langle S \rangle_{\ex}
\subset \dD$
the smallest extension closed subcategory of $\dD$
which contains $S$. For an abelian category $\aA$
and a set of objects $S\subset \aA$, the subcategory 
$\langle S \rangle_{\ex}\subset \aA$ is also defined to be 
the smallest extension closed subcategory of $\aA$
which contains $S$.

\section{Triangulated category of D0-D6 bound state}
Let $X$ be a smooth projective Calabi-Yau 3-fold 
over $\mathbb{C}$, i.e. 
$$K_X=\wedge^{3}T_{X}^{\ast}\cong \oO_X, \quad H^1(\oO_X)=0.$$
We denote by $\Coh_{0}(X)$ the subcategory of $\Coh(X)$, 
defined by 
$$\Coh_{0}(X)=\{E\in \Coh(X) : \dim \Supp(E)=0\}.$$
In this section, we study the triangulated 
subcategory of $D^b(\Coh(X))$
generated by $\oO_X$ and objects in $\Coh_{0}(X)$, 
$$\dD_{X} \cneq \langle \oO_X, \Coh_{0}(X)\rangle_{\tr}
\subset D^b(\Coh(X)).$$
The triangulated category $\dD_X$ is called the 
category of \textit{D0-D6 bound state} in~\cite[Paragraph~6.5]{K-S}. 

\subsection{t-structure on $\dD_X$}
Here we construct the heart of a bounded 
t-structure on $\dD_X$. 
The readers can refer~\cite[Section~4]{GM}
for the notion of bounded t-structures and their hearts. 
\begin{lem}
There is the heart of a bounded t-structure 
$\aA_X \subset \dD_X$, written as 
\begin{align}\label{written}
\aA_X=\langle \oO_X, \Coh_{0}(X)[-1]\rangle_{\ex}.
\end{align}
\end{lem}
\begin{proof}
Let $\fF$ be the subcategory of $\Coh(X)$, defined by 
$$\fF\cneq \{ E\in \Coh(X) : \Hom(F, E)=0 \mbox{ for any }F\in 
\Coh_{0}(X)\}.$$
Then $(\Coh_{0}(X), \fF)$ is a torsion pair on $\Coh(X)$.
(cf.~\cite{HRS}.)
Let $\aA^{\dag}\subset D^b(\Coh(X))$ be the associated 
tilting, 
$$\aA^{\dag}=\langle \fF, \Coh_{0}(X)[-1]\rangle_{\ex}.$$
Note that $\aA^{\dag}$ is the heart of a bounded 
t-structure on $D^b(\Coh(X))$. 
(cf.~\cite[Proposition~2.1]{HRS}.)
It is easy to see the following. 
\begin{itemize}
\item We have 
\begin{align}\label{cond1}
\aA^{\dag}\cap D^b(\Coh_{0}(X))=\Coh_{0}(X)[-1],
\end{align}
in $D^b(\Coh(X))$. In particular the LHS of (\ref{cond1})
is the heart of a bounded t-structure on $D^b(\Coh_{0}(X))$. 
\item For any $F\in \Coh_{0}(X)$, we have 
$$\Hom(\oO_X, F[-1])=\Hom(F[-1], \oO_X)=0, $$
by the Serre duality. 
\end{itemize}
Then we can apply~\cite[Proposition~3.3]{Tcurve1}, 
and conclude that 
$\aA_X \cneq \aA^{\dag}\cap \dD_X$ is the heart of a bounded 
t-structure on $\dD_X$, satisfying (\ref{written}). 
\end{proof}
The abelian category $\aA_X \subset \dD_X$ 
is described 
in a simpler way, as follows. 
\begin{prop}
The abelian category $\aA_X$
given by (\ref{written}) 
is equivalent to the abelian category of triples
\begin{align}\label{cat:triple}
(\oO_X^{\oplus r}, F, s),
\end{align} where $r$ is an integer, 
$F\in \Coh_{0}(X)$ and $s\colon \oO_{X}^{\oplus r} \to F$
is a morphism in $\Coh(X)$. The set of morphisms 
from $(\oO_X^{\oplus r}, F, s)$ to $(\oO_{X}^{\oplus r'}, F', s')$
is
given by the commutative diagrams, 
\begin{align}\label{comdiag}
\xymatrix{
\oO_{X}^{\oplus r} \ar[r]^{s}\ar[d]_{\alpha} & F \ar[d]_{\beta} \\
\oO_{X}^{\oplus r'} \ar[r]^{s'}& F'.
}
\end{align} 
The equivalence is given by sending a triple 
$E=(\oO_X^{\oplus r}, F, s)$
to the two term complex 
 \begin{align}\label{twoterm}
\Phi(E)=\cdots \to 0 \to \oO_X^{\oplus r} \stackrel{s}{\to}
 F \to 0 \to \cdots \in \aA_X,
\end{align}
where $\oO_X^{\oplus r}$ is located in degree zero.
\end{prop}
\begin{proof}
For a triple
$E=(\oO_X^{\oplus r}, F, s)$ as in (\ref{cat:triple}), 
note that the two term complex $\Phi(E)$
given by (\ref{twoterm}) fits into 
the exact sequence in $\aA_X$, 
\begin{align*}
0 \lr F[-1] \lr \Phi(E) \lr \oO_{X}^{\oplus r} \lr 0.
\end{align*}
Let us consider a diagram (\ref{comdiag}). 
Since $\Hom(\oO_X^{\oplus r}, F'[-1])=0$, there is a unique 
morphism $\gamma \colon \Phi(E)\to \Phi(E')$ which 
fits into the commutative diagram, 
\begin{align}\label{comdia2}
\xymatrix{
0 \ar[r] & F[-1] \ar[r]\ar[d]_{\beta[-1]} & \Phi(E) \ar[r]\ar[d]_{\gamma}
 & \oO_{X}^{\oplus r} \ar[r]\ar[d]_{\alpha} & 0 \\
0 \ar[r] & F[-1] \ar[r] & \Phi(E) \ar[r]
 & \oO_{X}^{\oplus r} \ar[r] & 0.}
\end{align}
Hence $E \mapsto \Phi(E)$ is a functor from the category of 
triples (\ref{cat:triple}) to $\aA_X$. Using the diagram (\ref{comdia2}) 
and $\Hom(F[-1], \oO_{X}^{\oplus r'})=0$, 
it is easy to see that $\Phi$ is fully faithful. 
Hence it suffices to show that $\Phi$ is essentially surjective. 

Let us take an object $M\in \aA_X$. 
By (\ref{written}), 
there is a filtration in $\aA_X$, 
$$M_0 \subset M_1 \subset \cdots \subset M_k=M,$$
such that each subquotient $N_i=M_i/M_{i-1}$ is 
isomorphic to $\oO_X$ or an object in $\Coh_{0}(X)[-1]$. 
We show that each $M_j$ is quasi-isomorphic 
to a two term complex (\ref{twoterm}) 
by the induction on $j$. 
The case of $j=0$ is obvious. Suppose that 
$M_{j-1}$ is isomorphic to a 
two term complex $(\oO_X^{\oplus r}\stackrel{s}{\to}F)$
for $F\in \Coh_{0}(X)$. There are two cases. 
\begin{case}\emph{$N_j$ is isomorphic to $\oO_X$.} 
\end{case}
In this case, we have the commutative diagram, 
$$\xymatrix{
& \oO_X[-1] \ar[d]\ar[rd]^{0} & \\
F[-1] \ar[r] &M_{j-1} \ar[r] & \oO_{X}^{\oplus r}, 
}$$
since $H^1(\oO_X)=0$. Taking the cones, we obtain 
the distinguished triangle, 
$$F[-1] \lr M_{j} \lr \oO_X^{\oplus r+1}.$$
Therefore $M_{j}$ is quasi-isomorphic to 
a two term complex $(\oO_X^{\oplus r+1} \to F)$. 
\begin{case}\emph{$N_j$ is isomorphic to $F'[-1]$
for $F'\in \Coh_{0}(X)$.} 
\end{case}
In this case, we have the commutative diagram, 
$$\xymatrix{
& F'[-2]\ar[ld]\ar[d] &\\
F[-1] \ar[r] & M_{j-1} \ar[r] & \oO_X^{\oplus r}, }$$
since $\Hom(F'[-2], \oO_X^{\oplus r})=0$. 
Taking the cones, we obtain the distinguished triangle, 
$$F''[-1] \lr M_{j} \lr \oO_X^{\oplus r}.$$
Here $F''$ fits into the exact sequence of sheaves
$0 \to F\to F'' \to F' \to 0$, hence $F''\in \Coh_{0}(X)$. 
Then $M_j$ is quasi-isomorphic to a two term complex 
$(\oO_X^{\oplus r} \to F'')$. 
\end{proof}
In what follows, we write an object $E\in \aA_X$
as a two term complex $(\oO_X^{\oplus r} \to F)$
occasionally. 
We set $S_0, S_x \in \aA_X$ for $x\in X$
as follows, 
\begin{align}\label{S0x}
S_0=(\oO_X \to 0), \quad S_x=(0 \to \oO_x).
\end{align}
The following lemma is obvious. 
\begin{lem}\label{simple}
An object $E\in \aA_X$
is simple if and only if 
$E$ is isomorphic to $S_0$ or $S_x$
for $x\in X$. 
Any objects in $\aA_X$ is written as successive extensions of 
these simple objects. 
\end{lem}
\subsection{Stability condition on $\aA_X$}
Here we discuss stability conditions on $\aA_X$, 
and the associated (semi)stable objects in $\aA_X$. 
The stability condition discussed here is 
based on the notion of stability 
conditions on triangulated categories by Bridgeland~\cite{Brs1}.

Let $\aA_X \subset \dD_X$ be the abelian category 
given by (\ref{written}).
We set $\Gamma=\mathbb{Z}\oplus \mathbb{Z}$
and a group homomorphism 
$$\cl \colon K(\aA_X) \to \Gamma, $$
by the following,
$$\cl \colon (\oO_X^{\oplus r} \to F) \mapsto
(r, \length F).$$
Also we denote by $\mathbb{H}\subset \mathbb{C}$
the upper half plane, 
$$\mathbb{H}=\{ z\in \mathbb{C} : \Imm z>0\}.$$
\begin{defi}\emph{
A \textit{stability condition} on $\aA_X$ is a group 
homomorphism $Z\colon \Gamma \to \mathbb{C}$, 
which satisfies 
$$Z(\cl(E)) \in \mathbb{H},$$
for any non-zero object $E\in \aA_X$. }
\end{defi}
In what follows, we write $Z(\cl(E))$ as $Z(E)$ for simplicity. 
\begin{rmk}
By Lemma~\ref{simple}, a group homomorphism $Z\colon \Gamma \to \mathbb{C}$
is a stability condition on $\aA_X$ if and only if 
$$Z(1, 0)\in \mathbb{H}, \quad Z(0, 1)\in \mathbb{H}.$$
In particular the set of stability conditions is 
parameterized by points in $\mathbb{H}^2$. 
\end{rmk}
\begin{rmk}
For a stability condition $Z\colon \Gamma \to \mathbb{C}$
on $\aA_X$, the pair $(Z, \aA_X)$ determines 
a stability condition on $\dD_X$ in the sense of
Bridgeland~\cite{Brs1}. 
\end{rmk}
The notion of (semi)stable objects are defined 
as follows. 
\begin{defi}\emph{
Let $Z\colon \Gamma \to \mathbb{C}$ be a stability 
condition on $\aA_X$. We say $E\in \aA_X$ is 
$Z$-\textit{semistable} (resp.~\textit{stable})
 if for any non-zero proper subobject
$0\subsetneq F\subsetneq E$ in $\aA_X$, the 
following inequality holds, }
\begin{align*}
&\arg Z(F) <\arg Z(E), \quad
(\mbox{\emph{resp}}.~\arg Z(F)
\le \arg Z(E).)
\end{align*}
\end{defi}
\subsection{Semistable objects in $\aA_X$}
We fix three stability conditions on $\aA_X$, 
\begin{align}\label{Zast}
Z_{\ast} \colon \Gamma \to \mathbb{C}, \quad 
\ast=\pm, 0
\end{align}
satisfying the following, 
\begin{align*}
&\arg Z_{+}(1, 0)>\arg Z_{+}(0, 1), \\
&\arg Z_{0}(1, 0)=\arg Z_{0}(0, 1), \\
&\arg Z_{-}(1, 0)<\arg Z_{-}(0, 1).
\end{align*}
The set of $Z_{\ast}$-(semi)stable objects are 
characterized as follows. 
\begin{prop}\label{prop:char}
(i) An object $E\in \aA_X$ is $Z_{-}$-(semi)stable 
if and only if $E$ is isomorphic to 
\begin{align}\label{char}
(\oO_X^{\oplus r} \to 0) \quad \mbox{ or } \quad (0 \to F),
\end{align}
for $r\in \mathbb{Z}$ and $F\in \Coh_{0}(X)$. 
(resp.~isomorphic to 
$S_0$ or $S_x$ for $x\in X$, given in (\ref{S0x}).) 

(ii) Any object in $\aA_X$ is $Z_{0}$-semistable, and 
$E\in \aA_X$ is $Z_{0}$-stable if and only if 
$E$ is isomorphic to $S_0$ or $S_x$ for $x\in X$. 

(iii) An object $E\in \aA_X$ is $Z_{+}$-(semi)stable 
if and only if $E$ is isomorphic to (\ref{char}), 
(resp. $S_0$ or $S_x$ for $x\in X$,) or 
isomorphic to $(\oO_X^{\oplus r} \stackrel{s}{\to} F)$
with $r>0$, $F\neq 0$, satisfying the following. 
\begin{itemize}
\item The image of the induced morphism between 
global sections, 
\begin{align}\label{V}
V=\Imm\{ H^{0}(s)\colon \mathbb{C}^{\oplus r} \to H^0(F)\},
\end{align}
is $r$-dimensional and generates $F$ as an $\oO_X$-module. 
\item For any non-zero proper subvector space $0\subsetneq W \subsetneq V$, 
the subsheaf $F_W \cneq \oO_X \cdot W \subset F$ satisfies 
\begin{align}\label{desire}
\frac{\length F_{W}}{\dim W}\ge \frac{\length F}{r}, \quad 
\left(resp.~\frac{\length F_{W}}{\dim W}> \frac{\length F}{r}.\right)
\end{align}
\end{itemize}
\end{prop}
\begin{proof}
(i) Take a non-zero object $E\in \aA_X$, 
which is isomorphic to $(\oO_X^{\oplus r} \stackrel{s}{\to} F)$
for $F\in \Coh_{0}(X)$. 
We have the exact sequence in $\aA_X$, 
\begin{align}\label{dest}0 \lr F[-1] \lr E \lr \oO_X^{\oplus r} \lr 0.
\end{align}
If $r\neq 0$ and $F\neq 0$, then we have 
$$\arg Z_{-}(F[-1])>\arg Z_{-}(E),$$
hence (\ref{dest}) destabilizes $E$. Therefore 
if $E$ is $Z_{-}$-semistable, 
we have $r=0$ or $F=0$. 
Furthermore if $E$ is $Z_{-}$-stable, 
$r=1$ or $\length F=1$ must hold.
Hence $E$ is isomorphic to 
$S_0$ or $S_x$ for $x\in X$. 
Conversely it is easy to see that 
objects in
(\ref{char}), (resp.~$S_0$, $S_x$ for $x\in X$,)
 are $Z_{-}$-semistable. (resp.~$Z_{-}$-stable.)
 
(ii) The proof of (ii) is obvious. 
 
(iii) Let us take
 a non-zero object $E=(\oO_X^{\oplus r} \stackrel{s}{\to} F)\in \aA_X$.
If $r=0$ or $F=0$, it is easy to see that $E$ is $Z_{+}$-semistable, 
and it is furthermore 
$Z_{+}$-stable if and only if 
$E$ is isomorphic to $S_0$ or $S_x$ for $x\in X$. 
Therefore we assume that $r\neq 0$ and $F\neq 0$. 

Suppose that $E$ is $Z_{+}$-(semi)stable, and take
$V\subset H^0(F)$ as in (\ref{V}).
If $\dim V<r$, then there is an injection $\oO_X \hookrightarrow E$
in $\aA_X$. Then we have 
$$\arg Z_{+}(\oO_X)>\arg Z_{+}(E).$$
This contradicts to that $E$ is $Z_{+}$-semistable, hence 
$V$ is $r$-dimensional. Furthermore if $V$ does not generate
$F$ as an $\oO_X$-module, there is a closed point $x\in X$
and a surjection $E\twoheadrightarrow \oO_x[-1]$
in $\aA_X$. Since 
$$\arg Z_{+}(E)>\arg Z_{+}(\oO_x),$$
this is a contradiction. 
Also take a subvector space $0\subsetneq W \subsetneq V$
and the subsheaf of $F$, $F_W=\oO_X \cdot W \subset F$. 
Then there is an injection in $\aA_X$, 
$$(\oO_X \otimes_{\mathbb{C}} W \twoheadrightarrow F_W) \hookrightarrow E, $$
hence the $Z_{+}$-(semi)stability implies
the desired inequality (\ref{desire}). 

Conversely suppose that $V$ is $r$-dimensional, 
$V$ generates $F$ as an $\oO_X$-module
 and the inequality (\ref{desire}) holds. 
Since $V$ generates $F$, 
the morphism $s\colon 
\oO_X^{\oplus r} \to F$ is surjective, 
and $E$ is a coherent sheaf. 
Take an injection in $\aA_X$, 
\begin{align}\label{inj}
E'=(\oO_X^{\oplus r'} \stackrel{s'}{\to} F')\hookrightarrow E.
\end{align}
If $r'=r$, then (\ref{inj}) is an isomorphism 
since $\oO_X^{\oplus r} \stackrel{s}{\to} F$ is surjective. 
If $r'=0$, then $\arg Z_{+}(E')<\arg Z(E)$ is obviously satisfied. 
Let us assume $0<r'<r$, and take $F''=\Imm s' \subset F'$.
Note that there are injections in $\aA_X$, 
$$E''=(\oO_X^{\oplus r'}\twoheadrightarrow F'')
 \hookrightarrow E' \hookrightarrow E.$$ 
 Since the cokernel of $E''\hookrightarrow E'$
 lies in $\Coh_0(X)[-1]$, we have 
\begin{align}\label{imp1}
\arg Z_{+}(E') \le \arg Z_{+}(E'').
\end{align}
Also since $V$ is $r$-dimensional, the inequality (\ref{desire})
implies 
\begin{align}\label{imp2}
\arg Z_{+}(E)>\arg Z_{+}(E''), \quad 
(\mbox{resp}.~\arg Z_{+}(E)\ge \arg Z_{+}(E'').)
\end{align}
By (\ref{imp1}) and (\ref{imp2}), the object $E$ is
$Z_{+}$-(semi)stable. 
\end{proof}
\begin{rmk}\label{rmk:nonpro}
By Proposition~\ref{prop:char} (iii), 
giving a $Z_{+}$-semistable $E\in \aA_X$
is equivalent to giving 
a pair $(F, V)$, where 
$F\in \Coh_{0}(X)$
and $V$ is a linear subspace $V\subset H^0(F)$
which generates $F$ as an $\oO_X$-module, and satisfying 
the stability condition (\ref{desire}). 
The notion of such pairs $(F, V)$ also makes sense for 
non-projective Calabi-Yau 3-fold $X$. 
\end{rmk}
\begin{exam}\emph{
(i) If $r=1$, then 
$(F, V)$ gives a $Z_{+}$-semistable object if and only 
if $V$ generates $F$ as an $\oO_X$-module. 
Suppose that $X=\mathbb{C}^3$.
The torus $T=\mathbb{G}_m^3$
acts on $X$, and the $T$-invariant 
pairs $(F, V)$ 
with $\length F=n$ 
bijectively corresponds to 3-dimensional 
partitions. For instance, 
the case of $n=3$ is as follows, 
\begin{align*}
F=\left\{\begin{array}{l}
\mathbb{C}\oplus \mathbb{C}x \oplus \mathbb{C}x^2 \\
\mathbb{C}\oplus \mathbb{C}y \oplus \mathbb{C}y^2 \\
\mathbb{C}\oplus \mathbb{C}z \oplus \mathbb{C}z^2 \\
\mathbb{C}\oplus \mathbb{C}x \oplus \mathbb{C}y \\
\mathbb{C}\oplus \mathbb{C}y \oplus \mathbb{C}z \\
\mathbb{C}\oplus \mathbb{C}y \oplus \mathbb{C}z
\end{array} \right. 
\supset
V=\left\{\begin{array}{l}
\mathbb{C} \\
\mathbb{C} \\
\mathbb{C} \\
\mathbb{C} \\
\mathbb{C} \\
\mathbb{C}
\end{array} \right. 
\end{align*}
Here $x, y, z$ are coordinates of $\mathbb{C}^3$.} 

 \emph{(ii)
Suppose that $X=\mathbb{C}^3$ and $(r, n)=(2, 3)$.
In the notation of (i), the $T$-fixed 
$Z_{+}$-semistable $(F, V)$ are classified as follows. 
\begin{align*}
F=\left\{\begin{array}{l}
\mathbb{C}\oplus \mathbb{C}x \oplus \mathbb{C}x^2 \\
\mathbb{C}\oplus \mathbb{C}y \oplus \mathbb{C}y^2 \\
\mathbb{C}\oplus \mathbb{C}z \oplus \mathbb{C}z^2 \\
\mathbb{C}x\oplus \mathbb{C}y \oplus \mathbb{C}xy \\
\mathbb{C}y\oplus \mathbb{C}z \oplus \mathbb{C}yz \\
\mathbb{C}x\oplus \mathbb{C}z \oplus \mathbb{C}xz
\end{array} \right. 
\supset
V=\left\{\begin{array}{l}
\mathbb{C}\oplus \mathbb{C}x \\
\mathbb{C}\oplus \mathbb{C}y \\
\mathbb{C}\oplus \mathbb{C}z \\
\mathbb{C}x \oplus \mathbb{C}y \\
\mathbb{C}y \oplus \mathbb{C}z \\
\mathbb{C}x \oplus \mathbb{C}z
\end{array} \right. 
\end{align*}
}
\end{exam}
\subsection{Moduli stacks}
Here we discuss the moduli 
stack of objects in $\aA_X$
and its substack of semistable object.  
For the notion of stacks, the readers can refer~\cite{GL}. 

Let $\oO bj(\aA_X)$ be the 2-functor, 
\begin{align*}
\oO bj(\aA_X)\colon \Sch/\mathbb{C} \to 
\groupoid,
\end{align*}
which sends a $\mathbb{C}$-scheme $S$
to the groupoid of objects $\eE \in D^b(X\times S)$, 
which is relatively perfect over $S$ and satisfies 
$\mathbb{L}i_{s}^{\ast}\eE \in \aA_X$
for any closed point $s\in S$. 
(See~\cite{LIE}.)
Here $i_{s}\colon X\times \{s\} \hookrightarrow X\times S$
is the inclusion. 
The 2-functor $\oO bj(\aA_X)$ forms a stack, and 
we have the decomposition,
$$\oO bj(\aA_X)=\coprod_{(r, n)\in \Gamma}\oO bj^{(r, n)}(\aA_X),$$
where $\oO bj^{(r, n)}(\aA_X) \subset \oO bj(\aA_X)$
is the substack of objects $E\in \aA_X$ with 
$\cl(E)=(r, n)$. 

Let us show that $\oO bj^{(r, n)}(\aA_X)$ is an
algebraic stack of finite type by describing 
it as a global quotient stack of 
the Quot scheme. 
For $(r, n) \in \Gamma$, 
recall that the Grothendieck's Quot scheme~\cite{Hu}
parameterizes isomorphism classes of quotients, 
$$\Quot^{(n)}(\oO_X^{\oplus r})=\{
\oO_X^{\oplus r}\stackrel{s}{\twoheadrightarrow}
 F : F\in \Coh_{0}(X), \length F=n\}/\cong.$$
Here two quotients $\oO_X^{\oplus r}\stackrel{s}{\twoheadrightarrow} F$
and $\oO_X^{\oplus r}\stackrel{s'}{\twoheadrightarrow} F'$
are isomorphic if and only if there is a commutative diagram, 
$$\xymatrix{
\oO_{X}^{\oplus r}\ar[r]^{s}\ar[d]_{\id} & F \ar[d]_{\cong} \\
\oO_X^{\oplus r}\ar[r]^{s'} & F'.
}$$
In particular there are no non-trivial automorphisms, and
the resulting moduli space $\Quot^{(n)}(\oO_X^{\oplus r})$
is a projective fine moduli scheme.
Note that there is a natural right $\GL(r, \mathbb{C})$-action 
on $\Quot^{(n)}(\oO_X^{\oplus r})$, 
given by 
$$(\oO_X^{\oplus r}\stackrel{s}{\twoheadrightarrow}
F)\cdot g=(\oO_X^{\oplus r}\stackrel{s \cdot g}{\twoheadrightarrow}F).$$

 We set 
$$U^{(n)}=\{(\oO_X^{\oplus n}\stackrel{s}{\twoheadrightarrow} F)\in 
\Quot^{(n)}(\oO_X^{\oplus n}) \mid
H^0(s) \colon \mathbb{C}^n \stackrel{\cong}{\to} H^0(F)\}.$$
It is easy to see that $U^{(n)}$ is an open substack 
of $\Quot^{(n)}(\oO_X^{\oplus n})$. 
For an object $F\in \Coh_{0}(X)$ with $\length F=n$, 
let us choose an isomorphism $\mathbb{C}^{n}\cong H^{0}(F)$. 
By applying $\otimes_{\mathbb{C}}\oO_X$ and 
composing the natural surjection, 
$$ \oO_X^{\oplus n} \cong H^0(F)\otimes_{\mathbb{C}}
 \oO_X \twoheadrightarrow F, $$
we obtain a point in $U^{(n)}$.  
Such a point is obtained up to a choice of
an isomorphism 
$\mathbb{C}^{n}\cong H^{0}(F)$, hence 
$\oO bj^{(0, n)}(\aA_X)$ is constructed as the quotient stack, 
$$\oO bj^{(0, n)}(\aA_X)=[U^{(n)}/\GL(n, \mathbb{C})].$$
For $r>0$, the moduli stack $\oO bj^{(r, n)}(\aA_X)$ is constructed
as follows. Let $\qQ \in \Coh(U^{(n)}\times X)$
be an universal quotient sheaf restricted to $U^{(n)}$, 
and $\pi_{U}\colon U^{(n)}\times X\to U^{(n)}$ the projection. 
We construct the affine bundle $U^{(r, n)} \to U^{(n)}$ as
\begin{align}\label{Urn}
U^{(r, n)}=\sS pec_{\oO_{U^{(n)}}}\Sym^{\bullet}
(\pi_{U\ast}\qQ^{\oplus r})^{\ast}\to U^{(n)}.
\end{align}
It is easy to see that $U^{(r, n)}$ represents the functor 
sending a $\mathbb{C}$-scheme $S$ to the set of 
isomorphism classes of the diagram, 
\begin{align}\label{Urn2}
\oO_{S\times X}^{\oplus n}\twoheadrightarrow \fF 
\leftarrow \oO_{S\times X}^{\oplus r},
\end{align}
where $\fF$ is a coherent sheaf on $S\times X$
flat over $S$, and the induced quotient $\oO_X^{\oplus n}
\to \fF|_{\{s\} \times X}$ for each closed point $s\in S$ determines a point 
in $U^{(n)}$. There is a right 
$\GL(r, \mathbb{C})$-action on $U^{(r, n)}$
along the fibers of the morphism (\ref{Urn}), 
acting on the right arrow of (\ref{Urn2}). 
Also the right 
$\GL(n, \mathbb{C})$-action on $U^{(n)}$ naturally lifts to 
the right action on $U^{(r, n)}$, and these actions commute. 
Hence there is a right
$G^{(r, n)}\cneq \GL(r, \mathbb{C})\times \GL(n, \mathbb{C})$-action 
on $U^{(r, n)}$, and 
the moduli stack $\oO bj^{(r, n)}(\aA_X)$ can be constructed as 
\begin{align}\label{Inpati}
\oO bj^{(r, n)}(\aA_X)
=[U^{(r, n)}/G^{(r, n)}].
\end{align}
In particular $\oO bj^{(r, n)}(\aA_X)$ is an algebraic 
stack of finite type over $\mathbb{C}$. 
\begin{prop}\label{loc.dis}
For any $p\in U^{(r, n)}$,
there is a $G^{(r, n)}$-invariant analytic open subset 
$p\in U_p \subset U^{(r, n)}$, a $G^{(r, n)}$-equivariant 
embedding $U_p \subset M_p$ for a complex manifold with a right
$G^{(r, n)}$-action, and a $G^{(r, n)}$-invariant 
holomorphic function $f_p \colon M_p \to \mathbb{C}$ such that 
$$U_p=\{z\in M_p : df_{p}(z)=0\}.$$
\end{prop}
\begin{proof}
Suppose that $p\in U^{(r, n)}$ corresponds to a diagram, 
$$\oO_X^{\oplus n} \twoheadrightarrow F \leftarrow \oO_X^{\oplus r},$$
such that $F\in \Coh_0(X)$ decomposes as 
$$F=\bigoplus_{i=1}^{k}F_i, \quad \Supp(F_i)=\{x_i\},
\quad \length F_i=n_i,$$
for distinct closed points $x_1, x_2, \cdots x_i \in X$
and $n_i \in \mathbb{Z}$. 
Let us take an analytic small open neighborhood 
$x_i \in V_i \subset X$ such that each $V_i$ is isomorphic
to $\mathbb{C}^3$ as a complex manifold, 
and $V_i \cap V_j=\emptyset$ for $i\neq j$. 
Note that we have
\begin{align}\label{pU}
p\in \{(\oO_X^{\oplus n} \twoheadrightarrow F' \leftarrow \oO_X^{\oplus r})
\in U^{(r, n)} : \Supp(F')\subset \coprod_{i}V_i\}, 
\end{align}
and define $p\in U_p\subset U^{(r, n)}$ to be the connected component 
of the RHS of (\ref{pU}), which contains $p$. 
Obviously $U_{p}$ is $G^{(r, n)}$-invariant analytic open 
subset of $U^{(r, n)}$. Restricting to each $V_i$, 
 giving a point on $U_p$
is equivalent to giving a collection of diagrams, 
\begin{align}\label{collect}
\oO_{V_i}^{\oplus n} \twoheadrightarrow F_i' \leftarrow \oO_{V_i}^{\oplus r}, 
\quad \length F_i'=n_i, 
\end{align} 
for each $1\le i\le k$
such that the induced morphism 
$$\mathbb{C}^{n}=H^{0}(\oO_X^{\oplus n})\to 
\bigoplus_{i=1}^{k}H^0(\oO_{V_i}^{\oplus n})
\to \bigoplus_{i=1}^{k}H^0(F_i'),$$
is an isomorphism. 
Since $V_i \cong \mathbb{C}^3$, giving such a collection 
of data (\ref{collect}) is equivalent to 
giving a point 
\begin{align}\label{MM}
\{(X_i, Y_i, Z_i, \{v_i^{(j)}\}_{j=1}^{n}, 
\{s_i^{(j)}\}_{j=1}^{r})\}_{i=1}^{k}
\in \prod_{i=1}^{k} M_{n_i}(\mathbb{C})^{\times 3}
\times (\mathbb{C}^{n_i})^{n}
\times (\mathbb{C}^{n_i})^{r},
\end{align}
satisfying 
\begin{align}\label{eq1}
&X_i Y_i=Y_i X_i, \quad X_i Z_i=Z_i X_i, \quad 
Y_i Z_i=Z_i Y_i, \quad 1\le i \le k, \\
\label{eq2}
&\det \left( v^{(1)}, v^{(2)}, 
\cdots, v^{(n)}\right)\neq 0.
\end{align}
Here $X_i, Y_i, Z_i$ are elements of $M_{n_i}(\mathbb{C})$, 
$v_i^{(j)}$, $s_i^{(j)}$ are elements of $\mathbb{C}^{n_i}$, 
and we have regarded 
$$v^{(j)}\cneq
\sum_{i=1}^{k}v_i^{(j)}\in \bigoplus_{i=1}^{k}\mathbb{C}^{n_i}=\mathbb{C}^n,$$
as a column vector of $M_n(\mathbb{C})$. 
We set $M_p$ to be an open subset of the RHS of (\ref{MM}), 
satisfying only (\ref{eq2}). 
Then the zero set of the equation (\ref{eq1})
is the critical locus of the 
holomorphic function $f_p \colon M_p \to \mathbb{C}$, 
$$f_p \left(\{(X_i, Y_i, Z_i, \{v_i^{(j)}\}_{j=1}^{n}, 
\{s_i^{(j)}\}_{j=1}^{r})\}_{i=1}^{k}\right)=
\sum_{i=1}^{k}\tr(X_i Y_i Z_i -Z_i Y_i X_i).$$
Obviously $G^{(r, n)}$ acts on $M_p$ from the right, 
$f_p$ is $G^{(r, n)}$-invariant, and there is a 
$G^{(r, n)}$-equivariant isomorphism between $U_p$ 
and $\{df_{p}=0\}\subset M_p$. 
\end{proof}
Let $Z\colon \Gamma \to \mathbb{C}$ be a 
stability condition on $\aA_X$. 
Let
\begin{align}\label{emb}
\mM^{(r, n)}(Z) \subset \oO bj^{(r, n)}(\aA_X),
\end{align}
be the substack of $Z$-semistable objects $E\in \aA_X$
with $\cl(E)=(r, n)$. 
By Proposition~\ref{prop:char}, we have 
\begin{align*}
\mM^{(r, n)}(Z_{-})&=\left\{\begin{array}{cc}
\oO bj^{(r, n)}(\aA_X), & r=0 \mbox{ or }n=0,\\
\emptyset, & \mbox{ otherwise.}
\end{array}\right. \\
\mM^{(r, n)}(Z_0)&=\oO bj^{(r, n)}(\aA_X).
\end{align*}
Here $Z_{\ast}$ is given by (\ref{Zast}). 
The moduli stack $\mM^{(r, n)}(Z_{+})$ 
is described as follows. 
\begin{lem}\label{Qrn}
There is a $\GL(r, \mathbb{C})$-invariant
Zariski open subset $Q^{(r, n)}\subset \Quot^{(n)}(\oO_X^{\oplus r})$
such that 
$$\mM^{(r, n)}(Z_{+})=[Q^{(r, n)}/\GL(r, \mathbb{C})].$$
\end{lem}
\begin{proof}
Let $\widetilde{U}^{(r, n)}\subset U^{(r, n)}$
be the open subset corresponding to diagrams, 
$$\oO_X^{\oplus n} \twoheadrightarrow F 
\stackrel{s}{\leftarrow} \oO_X^{\oplus r}, $$
such that $s$ is surjective. Then the action of 
the subgroup $\{\id\}\times \GL(n, \mathbb{C})\subset G^{(r, n)}$
on $U^{(r, n)}$ is free, and the 
quotient space is 
$$\widetilde{U}^{(r, n)}/\GL(n, \mathbb{C})=
\Quot^{(n)}(\oO_X^{\oplus r}).$$
We set $Q^{(r, n)}\subset \Quot^{(n)}(\oO_X^{\oplus r})$
to be the subset corresponds to quotients
$\oO_X^{\oplus r} \stackrel{s}{\twoheadrightarrow}
 F$ such that 
the associated two term complex 
$(\oO_X^{\oplus r} \stackrel{s}{\twoheadrightarrow} F)\in \aA_X$
is $Z_{+}$-semistable. The subset $Q^{(r, n)}$ is 
$\GL(r, \mathbb{C})$-invariant, and it is straightforward 
to see that $Q^{(r, n)}$ is open in $\Quot^{(n)}(\oO_X^{\oplus r})$. 
(e.g. use the arguments of the openness of stability 
in~\cite[Theorem~3.20]{Tst3}.)
By (\ref{Inpati}), the quotient stack of 
$Q^{(r, n)}$ by the action of $\GL(r, \mathbb{C})$ 
coincides with the desired stack $\mM^{(r, n)}(Z_{+})$. 
\end{proof}
\section{Hall algebras and Donaldson-Thomas invariants}\label{sec:Hall}
In this section, we review the result of Joyce-Song~\cite{JS}
applied in our abelian category $\aA_X$. 
\subsection{Notation}
In this subsection, we introduce some 
notation on algebraic groups, 
following~\cite{Joy5}. 
Let $G$ be an affine algebraic group over $\mathbb{C}$
with maximal torus $T^{G}$. We say $G$ is \textit{special}
if every principal $G$-bundles over $\mathbb{C}$ is locally 
trivial in the Zariski topology. 
For a subset $S\subset G$, the \textit{normalizer}
$N_G(S)$ and the \textit{centralizer} $C_G(S)$ 
of $S$ in $G$ are
\begin{align*}
N_G(S)&=\{g\in G : g^{-1}Sg=S\}, \\
C_{G}(S)&=\{g\in G : sg=gs \mbox{ for all }s\in S\},
\end{align*}
and the centre of $G$ is $C(G)\cneq C_{G}(G)$.  
For a subset $S\subset T^{G}$, note that 
$S\subset T^{G}\cap C(C_{G}(S))$. 
\begin{defi}\emph{{\bf \cite[Definition~5.5]{Joy5}}}
\emph{
We define the set $\qQ(G, T^G)$ to be the 
set of closed $\mathbb{C}$-subgroups $S$ of $T^G$, satisfying 
$$S=T^{G}\cap C(C_{G}(S)).$$
We say $G$ is very special if any $S\in \qQ(G, T^G)$ is 
special. }
\end{defi}
It is shown in~\cite[Lemma~5.6]{Joy5}
that $\qQ(G, T^{G})$ is a finite set, 
and any $S\in \qQ(G, T^{G})$ is written as an intersection 
of $T^{G}$ and $C_G(\{t_i\})$ for a finite set of points 
$t_1, \cdots, t_k \in G$. 
\begin{exam}\label{Ex1}\emph{
Suppose that $G=\GL(2, \mathbb{C})$,
and $\mathbb{G}_m^{2}\cong 
T^{G}\subset G$ is the subgroup of diagonal matrices. 
Then $\qQ(G, T^{G})$ consists of 
$T^{G}$ and the following subgroup. 
(cf.~\cite[Example~5.7]{Joy5}.)
\begin{align}\label{diag}
\mathbb{G}_m \cong \left\{ \left( \begin{array}{cc} 
t & 0 \\
0 & t
\end{array}\right) :
 t\in \mathbb{C}^{\ast} \right\}\subset T^{G}.
 \end{align}
 In particular $\GL(2, \mathbb{C})$ is a
 very special algebraic group.} 
\end{exam}
In~\cite{Joy5}, D.~Joyce introduces
an important rational number
$F(G, T^G, S)$ for a very special algebraic group $G$
and $S\in \qQ(G, T^G)$, as follows. 
\begin{defi}\emph{{\bf \cite[Definition~5.8]{Joy5}, 
\cite[Definition~6.7]{Joy5}}}\emph{
Let $G$ be a very special algebraic group. 
For $S\subset S'$ in $\qQ(G, T^G)$, we define 
$n_{T^G}^{G}(S, S')\in \mathbb{Z}$ to be}
$$n_{T^G}^{G}(S, S')=\sum_{S'\in A\subseteq \{S''\in \qQ(G, T^G) : 
S''\subset S'\}, \ \cap_{S''\in A}S''=S}
(-1)^{\lvert A \rvert -1},$$
\emph{and for $S\in \qQ(G, T^G)$, 
define $F(G, T^G, S)\in \mathbb{Q}$ by} 
$$F(G, T^G, S)=\lim_{t\to 1}
\sum_{\begin{subarray}{c}S'\in \qQ(G, T^G) \\
S\subset S'
\end{subarray}}\left\lvert \frac{N_G(T^G)}{C_G(S')\cap N_G(T^G)}\right\rvert ^{-1}
\cdot n_{T^G}^G(S, S') \frac{P_t(S)}{P_t(C_G(S'))}.$$
\end{defi}
Here for a quasi-projective $\mathbb{C}$-variety $Y$, 
the \textit{virtual Poincar\'e polynomial}
$P_t(Y)\in \mathbb{Q}[t]$ is defined by 
$$P_t(Y)=\sum_{j, k\ge 0}\dim (-1)^{k}W_j(H_{c}^{k}(Y, \mathbb{C}))t^{j},$$
where $W_{\ast}(H_{c}^{k}(Y, \mathbb{C}))$ is the weight
filtration on the compact support cohomology 
group $H_{c}^{k}(Y, \mathbb{C})$ introduced by Deligne. 
The existence of the limit $t\to 1$ is proved 
in~\cite[Theorem~6.6]{Joy5}.

\begin{exam}\label{exam:F}
\emph{
For $G=\GL(2, \mathbb{C})$, it is easy to 
calculate $F(G, T^G, S)$ as follows. 
(cf.~Example~\ref{Ex1}, \cite[Paragraph~6.2]{Joy5}.)}
\begin{align*}
F(G, T^G, T^G)=\frac{1}{2}, \quad
F(G, T^G, \mathbb{G}_m)=-\frac{3}{4}.
\end{align*}\emph{
Here $\mathbb{G}_m \subset T^G$ is given by 
(\ref{diag}).}
\end{exam}
\subsection{Hall algebra}
Let $X$ be a smooth projective Calabi-Yau 
3-fold over $\mathbb{C}$, and 
$\aA_X \subset D_X$ the abelian subcategory 
given by (\ref{written}). 
Here we introduce the Hall algebra
based on the algebraic 
stack $\oO bj(\aA_X)$, 
following~\cite[Definition~6.8]{Joy5}.
\begin{defi}\emph{
We define the $\mathbb{Q}$-vector space $\hH(\aA_X)$
to be spanned by symbols, 
$$[\xX \stackrel{f}{\to} \oO bj(\aA_X)], $$
where $\xX$ is an algebraic stack of 
finite type with affine geometric 
stabilizers, and $f$ is a morphism of stacks, 
with relations as follows. }
\begin{itemize}
\item \emph{For a closed substack $\yY \subset \xX$
and $\uU=\xX \setminus \yY$, we have} 
$$[\xX \stackrel{f}{\to} \oO bj(\aA_X)]
=[\yY \stackrel{f|_{\yY}}{\to} \oO bj(\aA_X)]
+[\uU \stackrel{f|_{\uU}}{\to} \oO bj(\aA_X)].$$
\item \emph{For a quasi-projective 
$\mathbb{C}$-variety $U$, we have} 
$$[\xX \times U 
 \stackrel{\pi_{\xX}\circ f}{\to} \oO bj(\aA_X)]
 =\chi(U)[\xX \stackrel{f}{\to} \oO bj(\aA_X)].$$
 \emph{Here $\pi_{\xX}\colon \xX\times U \to \xX$ is the 
 projection, and $\chi(U)=P_{t}(U)|_{t=1}\in \mathbb{Z}$.} 
 \item \emph{Let $U$ be a quasi-projective 
 $\mathbb{C}$-variety and $G$ a very special 
 algebraic group, which acts on $U$
 with maximal torus $T^G$. Then we have
 \begin{align}
 \label{relHall}[[U/G] \stackrel{f}{\to} \oO bj(\aA_X)]
 =\sum_{S\in \qQ(G, T^G)}F(G, T^G, S)
 [[U/S] \stackrel{f\circ \tau^{S}}{\to}\oO bj(\aA_X)].
 \end{align}
 Here $\tau^{S}\colon [U/S] \to [U/G]$ is a natural morphism.} 
\end{itemize}
\end{defi}
We denote by $\eE x(\aA_X)$ the stack of short
exact sequences in $\aA_X$. There are 
morphisms of stacks, 
$$p_i \colon \eE x(\aA_X) \lr \oO bj(\aA_X),$$
sending a short exact sequence $0\to A_1 \to A_2 \to A_3 \to 0$
to objects $A_i$ respectively. 
There is an associative product 
on $\hH(\aA_X)$ based on 
Ringel-Hall algebras, defined by 
$$[\xX \stackrel{f}{\to} \oO bj(\aA_X)]
\ast [\yY \stackrel{g}{\to} \oO bj(\aA_X)]
=[\zZ \stackrel{p_2 \circ h}{\to} \oO bj(\aA_X)],$$
where the morphism $h$ fits into the Cartesian square, 
$$\xymatrix{
\zZ \ar[r]^{h}\ar[d] & \eE x(\aA_X) \ar[r]^{p_2}\ar[d]^{(p_1, p_3)}
 & \oO bj(\aA_X). \\
 \xX \times \yY \ar[r]^{f\times g} \ar[r] & \oO bj(\aA_X)^{\times 2}
 }$$ 
We have the following. 
\begin{thm}\emph{\cite[Theorem~5.2]{Joy2}}
The $\ast$-product is well-defined and
associative with unit given by $[\Spec \mathbb{C} \to 
\oO bj(\aA_X)]$ which corresponds to $0\in \aA_X$. 
\end{thm}
\subsection{Donaldson-Thomas invariant}
Let $Z\colon \Gamma \to \mathbb{C}$ be a 
stability condition on $\aA_X$. 
The embedding of the algebraic stack (\ref{emb}) 
defines an element 
$$\delta^{(r, n)}(Z)=[\mM^{(r, n)}(Z)\subset \oO bj(\aA_X)]
\in \hH(\aA_X).$$
In order to define counting invariants of $Z$-semistable
objects, we want to 
take a (weighted) Euler characteristic 
of the moduli stack 
$\mM^{(r, n)}(Z)$. However 
in general, geometric points on the 
moduli stack $\mM^{(r, n)}(Z)$ 
have non-trivial stabilizers, hence 
its Euler characteristic does not make sense. 
Instead 
we take the `logarithm' of $\delta^{(r, n)}(Z)$ in 
$\hH(\aA_X)$ to kill non-trivial stabilizers.  
\begin{defi}\emph{{\bf \cite[Definition~3.18]{Joy4}}}
\emph{We define $\epsilon^{(r, n)}(Z)\in \hH(\aA_X)$ 
to be} 
\begin{align}\label{vare}
\epsilon^{(r, n)}(Z)=\sum_{\begin{subarray}{c}
l\ge 0, \ (r_1, n_1)+\cdots +(r_l, n_l)=(r, n), \\
Z(r_i, n_i)\in \mathbb{R}_{>0}
Z(r, n) \emph{ for all }i.\end{subarray}}
\frac{(-1)^{l-1}}{l}\delta^{(r_1, n_1)}(Z)\ast \cdots \ast 
\delta^{(r_l, n_l)}(Z).
\end{align}
\end{defi}
Since $\delta^{(r, n)}(Z)$ is non-zero only 
if $r\ge 0$ and $n\ge 0$, the sum (\ref{vare}) 
is a finite sum. 
Also if $r$ and $n$ are coprime, 
then $\epsilon^{(r, n)}(Z)=\delta^{(r, n)}(Z)$.
The important fact~\cite[Corollary~5.10]{Joy2}, \cite[Theorem~8.7]{Joy3}
 is that $\epsilon^{(r, n)}(Z)$
is supported on 
`virtual indecomposable objects', and 
written as 
\begin{align}\label{Ui}
\epsilon^{(r, n)}(Z)=\sum_{i=1}^{m}c_i[U_i 
\times [\Spec \mathbb{C}/\mathbb{G}_m] \stackrel{f_i}{\to}
\oO bj(\aA_X)],
\end{align} 
for quasi-projective $\mathbb{C}$-varieties 
$U_1, \cdots, U_m$, and $c_1, \cdots, c_m \in \mathbb{Q}$. 
Now the (weighted) Euler characteristic of $\epsilon^{(r, n)}(Z)$
makes sense. 
\begin{defi}\emph{
Suppose that $\epsilon \in \hH(\aA_X)$ is written as 
\begin{align}\label{Ui2}
\epsilon=
\sum_{i=1}^{m}c_i[U_i 
\times [\Spec \mathbb{C}/\mathbb{G}_m] \stackrel{f_i}{\to}
\oO bj(\aA_X)]. 
\end{align}
For a constructible function $\mu \colon \oO bj(\aA_X) \to \mathbb{Z}$, 
we define $\chi(\epsilon, \mu)\in \mathbb{Q}$ to be 
\begin{align*}
\chi(\epsilon, \mu)=
\sum_{i=1}^{m}c_i \sum_{k\in \mathbb{Z}}
\chi(f_i^{-1}\mu^{-1}(k)).
 \end{align*}}
\end{defi}

Next recall that for any $\mathbb{C}$-scheme $U$, K.~Behrend~\cite{Beh} 
associates a canonical constructible function
$\nu \colon U \to \mathbb{Z}, $
satisfying the following. 
\begin{itemize}
\item For $p\in U$, suppose that there is 
an analytic open neighborhood $p\in U_p$,
a complex manifold $M_p$ with $U_p\subset M_p$, 
and a holomorphic function $f_p \colon M_p \to \mathbb{C}$
such that $U_p=\{df_{p}=0\}$. Then 
$$\nu(p)=(-1)^{\dim M_p}(1-\chi(M_p(f_p))).$$
Here $M_p(f_p)$ is the Milnor fiber of $f_p$ at $p$. 
\item If $U$ has a symmetric perfect obstruction 
theory with zero dimensional virtual cycle 
$U^{\rm{vir}}$, we have 
$$\int_{U^{\rm{vir}}}1=\int_{U}\nu \ d\chi.$$
\end{itemize}
The notion of Behrend's constructible function 
can be easily extended to an arbitrary algebraic stack. 
(cf.~\cite[Proposition~4.4]{JS}.)
Hence we have the Behrend constructible function, 
$$\nu \colon \oO bj(\aA_X) \to \mathbb{Z}.$$
Explicitly using the notation of (\ref{Inpati}) and 
Proposition~\ref{loc.dis}, 
we have
$$\nu(p)=(-1)^{n+r+nr}(1-\chi(M_p(f_p))),$$
for $p\in U^{(r, n)}$. 
We then define $\DT(r, n)\in \mathbb{Q}$ as follows. 
(cf.~\cite[Definition~5.13]{JS}.)
\begin{defi}\label{def:DT}
\emph{
We define $\DT(r, n) \in \mathbb{Q}$ to be 
$$\DT(r, n)=\chi(\epsilon^{(r, n)}(Z_{+}), -\nu).$$}
\end{defi}
Here we need to change the sign of the Behrend function. 
This is basically because that the Behrend functions 
on the 
variety $M$ 
and on the stack $M\times [\Spec \mathbb{C}/\mathbb{G}_m]$
have the different sign.

\begin{rmk}\label{rmk:Mac}
(i) If $r=1$, then $\DT(1, n)$ coincides with 
the Donaldson-Thomas invariant counting points, 
studied and calculated in~\cite{MNOP}, \cite{Li}, \cite{BBr}, \cite{LP}. 
The result is 
$$\sum_{n\ge 0}\DT(1, n)q^n=M(-q)^{\chi(X)},$$
where $M(q)$ is the MacMahon function, 
\begin{align}\label{Mac}
M(q)=\prod_{k\ge 1}\frac{1}{(1-q^k)^k}.
\end{align}
(ii) For $n=0$, the invariant $\DT(r, 0)$ is 
easily shown to be (cf.~\cite[Example~6.2]{JS},~\cite[Paragraph~6.5]{K-S},)
\begin{align}\label{DTr0}
\DT(r, 0)=\frac{1}{r^2}.
\end{align}
(iii) For $r=0$, the invariant $\DT(0, n)$ is 
computed in~\cite[Paragraph~6.3]{JS}, \cite[Paragraph~6.5]{K-S}, 
\cite[Remark~8.13]{Tcurve1} using the 
wall-crossing formula. The result is 
\begin{align}\label{form:exp}
\exp\left(\sum_{n\ge 0}(-1)^{n-1}\DT(0, n)q^n\right)=M(-q)^{\chi(X)},
\end{align}
hence 
\begin{align}\label{DT0n}
\DT(0, n)=-\chi(X)\sum_{m\ge 0, m|n}\frac{1}{m^2}.
\end{align}
\end{rmk}
\subsection{Euler characteristic version}
In Section~\ref{sec:int}, we will also use the 
Euler characteristic version of 
counting invariants of $Z_{+}$-semistable 
objects in $\aA_X$, 
defined as follows. 
\begin{defi}\emph{
We define 
$\Eu(r, n)\in \mathbb{Q}$ to be 
\begin{align*}
\Eu(r, n)=\chi(\epsilon^{(r, n)}(Z_{+}), 1).
\end{align*}
Here $1$ is the constant constructible function on 
$\oO bj(\aA_X)$ which takes the value at $1$.}
\end{defi}
Similarly to $\DT(r, n)$, 
the invariant $\Eu(r, n)$ 
is already computed when $r=0$ or $n=0$. 
The result is (cf.~\cite[Example~6.2]{JS}, ~\cite[Remark~5.14]{Tcurve1},)
\begin{align}
\label{Eu1}
\Eu(r, 0)&=\frac{(-1)^{r-1}}{r^2}, \\
\label{Eu2}
\Eu(0, n)&=\chi(X) \sum_{m\ge 0, m|n}\frac{1}{m^2}.
\end{align}

\section{Computation of $\DT(2, n)$}
In this section, we deduce the generating 
series of $\DT(2, n)$ using the
wall-crossing formula
of DT-invariants. 
\subsection{Combinatorial coefficients}
In this subsection, we introduce
some notation which will be used 
in describing the wall-crossing formula. 
For $\Gamma=\mathbb{Z}\oplus \mathbb{Z}$, we set 
$$C(\Gamma)=\{(r, n) \in \Gamma \setminus \{0\} : r\ge 0, \ n\ge 0\}.$$
Define $\mu \colon C(\Gamma) \to \mathbb{Q}\cup \{\infty\}$ 
to be $\mu(r, n)=n/r$. 
\begin{defi}\label{def:s}
\emph{
For $l\ge 1$, we define the map 
$$s_{l} \colon C(\Gamma)^{l} \lr \{0, \pm 1 \}, $$
as follows. Suppose that $v_1, \cdots, v_l \in C(\Gamma)^{l}$ 
satisfies one of (a) or (b) for each $i$, }

\emph{(a) $\mu(v_i)>\mu(v_{i+1})$ and $\mu(v_1+\cdots +v_i)\ge  
\mu(v_{i+1}+\cdots +v_{l})$.} 

\emph{(b) $\mu(v_i)\le \mu(v_{i+1})$ and $\mu(v_1+\cdots +v_i)<
\mu(v_{i+1}+\cdots +v_{l})$.} 

\emph{Then $s_l(v_1, \cdots, v_l)=(-1)^{k}$, 
where $k$ is the number of $i=1, \cdots, l-1$
satisfying (b). 
Otherwise $s_l(v_1, \cdots, v_l)=0$.} 
\end{defi}
\begin{defi}\emph{
For $l\ge 1$, we define the map 
$$u_{l} \colon C(\Gamma)^{l} \lr \mathbb{Q},$$
as follows,} 
\begin{align}\notag
u_l(v_1, \cdots, v_l)=
\sum_{1\le l'' \le l' \le l}
\sum_{\begin{subarray}{c}\psi \colon \{1, \cdots, l\} 
\to \{1, \cdots, l'\}, \
\xi \colon \{1, \cdots, l'\} \to \{1, \cdots, l''\}, \\
\psi, \xi \emph{ are non-decreasing surjective maps,} \\
\mu(v_i)=\mu(v_j) \emph{ if }\psi(i)=\psi(j), \\
\mu(\sum_{k\in (\xi \circ \psi)^{-1}(i)}v_k)=
\mu(\sum_{k\in (\xi \circ \psi)^{-1}(j)}v_k) \emph{ for any }
i, j. \end{subarray}} \\ \label{sum:u}
\prod_{a=1}^{l''}s_{\lvert \xi^{-1}(a)\rvert}\left( \left\{
\sum_{k\in \psi^{-1}(j)}v_k \right\}_{j\in \xi^{-1}(a)}\right)
\frac{(-1)^{l''+1}}{l''}\prod_{b=1}^{l'}\frac{1}
{\rvert \psi^{-1}(b)\rvert !}.
\end{align}
\end{defi}
We introduce the notion of 
bi-colored weighted ordered vertex, as follows. 
\begin{defi}\emph{
We call a data 
\begin{align}\label{graph}
\Lambda=(V, \pi, v, \le ),
\end{align}
\emph{bi-colored weighted ordered vertex} if it satisfies the following.} 
\begin{itemize}
\item \emph{$V$ is a finite set.} 
\item \emph{$\pi \colon V \to \{\bullet, \circ\}$ is a map, 
where $\{\bullet, \circ\}$ is a set with two elements.} 
\item \emph{$v$ is a map $v\colon V\to \mathbb{Z}_{\ge 1}$.} 
\item \emph{$\le$ is a total order on $V$.}
\end{itemize}
\end{defi}
Let $\Lambda$ be a data (\ref{graph}) with $l=\lvert V \rvert$. 
The total order $\le$ on $V$ gives 
an identification between $V$ and $\{1, \cdots, l\}$. 
We set $V_{\bullet}$ and $V_{\circ}$ to be 
\begin{align*}
V_{\bullet}&=\{v\in V : \pi(v)=\bullet\}, \\
V_{\circ}&=\{v\in V : \pi(v)=\circ\}.
\end{align*}
We set
$v_i \in C(\Gamma)$ to be  
\begin{align*}
v_i=\left\{\begin{array}{cc}
(v(i), 0), & \mbox{ if }i\in V_{\bullet}, \\
(0, v(i)), & \mbox{ if }i\in V_{\circ}.
\end{array}\right. 
\end{align*}
We set $s(\Lambda)\in \{0, \pm 1 \}$ and 
$u(\Lambda)\in \mathbb{Q}$ to be 
\begin{align*}
s(\Lambda)=s_l(v_1, \cdots, v_l), \quad 
u(\Lambda)=u_l(v_1, \cdots, v_l).
\end{align*}
Also we set 
$$r(\Lambda)=\sum_{i\in V_{\bullet}}v(i), \quad 
n(\Lambda)=\sum_{i\in V_{\circ}}v(i).$$
We define $\DT(\Lambda)\in \mathbb{Q}$
and $\Eu(\Lambda)\in \mathbb{Q}$ to be 
\begin{align*}
\DT(\Lambda)&=\prod_{i\in V_{\bullet}}\DT(v(i), 0)
\prod_{i\in V_{\circ}}\DT(0, v(i)), \\
\Eu(\Lambda)&=\prod_{i\in V_{\bullet}}\Eu(v(i), 0)
\prod_{i\in V_{\circ}}\Eu(0, v(i)).
\end{align*}
\begin{defi}\emph{
Let $\Lambda=(V, \pi, v, \le)$ be a bi-colored
weighted ordered vertex. We define the 
set $\eE(\Lambda)$ to be the set of data 
$$(E, s, t),$$
satisfying the following.}
\begin{itemize}
\item \emph{$E$ is a finite set and $s$, $t$ are 
maps $E\to V$, i.e. the data $(V, E, s, t)$
determines a quiver. The geometric realization of 
this quiver is connected and simply connected.} 
\item \emph{For any $e\in E$, we have $\pi s(e) \neq \pi t(e)$.}
\item \emph{For any $e\in E$, we have $s(e)<t(e)$ with 
respect to the total order $\le$ on $V$.}
\end{itemize}
\end{defi}
For $(E, s, t)\in \eE(\Lambda)$, we set $E_{\bullet \to \circ}$ to be
$$E_{\bullet \to \circ}=\{ e\in E : \pi s(e)=\bullet\}.$$

\subsection{Combinatorial descriptions of $\DT(r, n)$, $\Eu(r, n)$}
Using the combinatorial data given in the 
previous subsection, we can
describe the invariant $\DT(r, n)$ as follows. 
\begin{thm}\label{thm:DTrn}
We have the following formula. 
\begin{align}\notag
\DT(r, n)&=\sum_{\begin{subarray}{c}
\Lambda=(V, \pi, v, \le) 
\emph{ is a bi-colored }
\\
\emph{ weighted ordered vertex with }\\
r(\Lambda)=r, \ n(\Lambda)=n.
\end{subarray} }
(-1)^{rn}
 u(\Lambda) 
\DT(\Lambda) \\ \label{DTrn}
& \qquad \qquad 
\left(-\frac{1}{2}\right)^{\lvert V \rvert -1}
\sum_{(E, s, t)\in \eE(\Lambda)}
(-1)^{\lvert E_{\bullet \to \circ}
 \rvert} 
\prod_{e\in E}v(s(e))v(t(e)).
\end{align}
\end{thm}
\begin{proof}
Let $\chi \colon \Gamma \times \Gamma \to \mathbb{Z}$
be 
$$\chi((r, n), (r', n'))=rn'-r'n.$$
For $E, F\in \aA_X$, we have 
\begin{align}\notag
\chi(\cl(E), \cl(F))&=\dim \Hom(E, F)-\dim \Ext ^1(E, F) \\
\label{RRSe}
&\qquad +\dim \Ext^1(F, E)-\dim \Hom(F, E),\end{align}
by the Riemann-Roch theorem and the Serre duality. 
The equation (\ref{RRSe}) provides
an analogue of~\cite[Equation~(39)]{JS}
and Proposition~\ref{loc.dis} provides an analogue of~\cite[Theorem~5.3]{JS}.
The proof of the Behrend function identity given 
in~\cite[Theorem~5.9]{JS}
depends on these two properties, hence
the analogue of~\cite[Theorem~5.9]{JS} also 
holds for our abelian category $\aA_X$. 
Then we can apply the proof of~\cite[Theorem~5.16]{JS}
for stability conditions $Z_{\pm} \colon \Gamma \to \mathbb{C}$, 
and obtain the same formula given in~\cite[Theorem~5.6]{JS}. 
Noting that 
\begin{align*}
\chi((r, 0), (r', 0))=0, \quad 
\chi((0, n), (0, n'))=0, 
\end{align*}
we obtain the formula (\ref{DTrn}). 
\end{proof}
The formula for $\Eu(r, n)$ is similarly 
obtained by using~\cite[Theorem~6.28]{Joy4} 
instead of~\cite[Theorem~5.16]{JS}. 
\begin{thm}\label{thm:Eurn}
We have the following formula. 
\begin{align}\notag
\Eu(r, n)&=\sum_{\begin{subarray}{c}
\Lambda=(V, \pi, v, \le) 
\emph{ is a bi-colored }
\\ 
\emph{ weighted ordered vertex with }\\
r(\Lambda)=r, \ n(\Lambda)=n.
\end{subarray} }
 u(\Lambda) 
\Eu(\Lambda) \\ \label{Eurn}
& \qquad \qquad 
\left(\frac{1}{2}\right)^{\lvert V \rvert -1}
\sum_{(E, s, t)\in \eE(\Lambda)}
(-1)^{\lvert E_{\bullet \to \circ}
 \rvert} 
\prod_{e\in E}v(s(e))v(t(e)).
\end{align}
\end{thm}
As a corollary, we have the following. 
\begin{cor}\label{cor:DT=Eu}
We have 
\begin{align}\label{DT=Eu}\DT(r, n)=(-1)^{rn+r-1}\Eu(r, n).
\end{align} 
\end{cor}
\begin{proof}
By the formulas (\ref{DTr0}), (\ref{DT0n}), 
(\ref{Eu1}) and (\ref{Eu2}), 
we have 
$$\DT(\Lambda)=(-1)^{\lvert V \rvert +r}\Eu(\Lambda),$$
for a bi-colored weighted ordered vertex $\Lambda=(V, \pi, v, \le)$
with $r(\Lambda)=r$. 
Applying formulas (\ref{DTrn}) and (\ref{Eurn}), we obtain 
(\ref{DT=Eu}). 
\end{proof}

\subsection{Computation of $s(\Lambda)$}
In this subsection, we compute $s(\Lambda)$ 
for a data (\ref{graph}) with $r(\Lambda)=2$. 
Let us take a data $(\ref{graph})$ 
with $\lvert V \rvert =l$ and 
\begin{align}\label{rn}
r(\Lambda)=2, \quad n(\Lambda)=n.
\end{align}
We fix an identification 
between $V$ and $\{1, \cdots, l\}$ induced by 
the total order $\le$.  
We denote by $\pi(\Lambda)$ the sequence of 
$\bullet$ and $\circ$, given by 
$$\pi(\Lambda)=\pi(1) \ \pi(2) \ \cdots \ \pi(l).$$
Note that we have $\lvert V_{\bullet} \rvert \le 2$. 
We first have the following lemma. 
\begin{lem}\label{lem:s=0}
Suppose that 
 $\pi(1)=\pi(2)=\circ$, i.e. $\pi(\Lambda)$
 is 
 $$\stackrel{1}{\circ} \ \stackrel{2}{\circ}
  \ \cdots \ \circ \ \bullet \cdots .$$
 Then $s(\Lambda)=0$. 
\end{lem}
\begin{proof}
Since $\mu(v_1)=\mu(v_2)$ and 
$\infty =\mu(v_1)>\mu(v_2+\cdots +v_l)$, 
$(v_1, \cdots, v_l)$ does not satisfy (a) nor (b) 
in Definition~\ref{def:s}. 
\end{proof}
Next we compute the case of $\lvert V_{\bullet}\rvert =1$. 
\begin{lem}\label{lem:coms}
Suppose that $\lvert V_{\bullet}\rvert =1$
with $s(\Lambda)\neq 0$. Then 
the value $s(\Lambda)$ is computed as follows. 
\begin{itemize}
 \item 
 Suppose that 
 $\pi(1)=\circ$, $\pi(2)=\bullet$ 
 and $\pi(i)=\circ$ for all $i\ge 3$, 
  i.e. $\pi(\Lambda)$
 is 
 \begin{align}\label{pi1}
 \stackrel{1}{\circ} \ \stackrel{2}{\bullet} \ \circ \ 
 \cdots \ \stackrel{l}{\circ}.
 \end{align}
 Then $s(\Lambda)=(-1)^{l}$. 
 \item Suppose that $\pi(1)=\bullet$ and $\pi(i)=\circ$
 for all $i\ge 2$, i.e. $\pi(\Lambda)$ is
 \begin{align}\label{pi2}
 \stackrel{1}{\bullet} \ \stackrel{2}{\circ} \ \circ \ 
 \cdots \ \stackrel{l}{\circ}.
 \end{align}
 Then $s(\Lambda)=(-1)^{l-1}$. 
\end{itemize}
\end{lem}
\begin{proof}
By Lemma~\ref{lem:s=0}, the sequence 
$\{ \pi(1), \pi(2), \cdots, \pi(l)\}$ is either 
(\ref{pi1}) or (\ref{pi2}). In case (\ref{pi1}), 
(resp.~(\ref{pi2}),) the condition (a) or (b) 
in Definition~\ref{def:s} is satisfied and 
the number of $1\le i \le l-1$ in which 
(b) holds is $l-2$. (resp.~$l-1$.)
\end{proof}
The case of $\lvert V_{\bullet}\rvert =2$ is 
computed as follows. 
\begin{lem}\label{lem:computed}
Suppose that $\lvert V_{\bullet}\rvert=2$
with $s(\Lambda)\neq 0$. Then
$l\ge 3$ and 
 $s(\Lambda)$ is computed as follows. 
\begin{itemize}
\item Suppose that $V_{\bullet}=\{1, 2\}$, i.e. 
$\pi(\Lambda)$ is 
\begin{align}\label{case1}
 \stackrel{1}{\bullet} \ \stackrel{2}{\bullet} \ \circ \ 
 \cdots \ \stackrel{l}{\circ}.
 \end{align}
 Then $s(\Lambda)=(-1)^{l-1}$. 
 \item Suppose that $V_{\bullet}=\{1, a\}$
 for $a\ge 3$, i.e. $\pi(\Lambda)$ is 
 \begin{align}\label{case2}
 \stackrel{1}{\bullet} \ \stackrel{2}{\circ} \ \cdots \ 
 \stackrel{a-1}{\circ} \ \stackrel{a}{\bullet} \ \stackrel{a+1}{\circ}
 \ \cdots 
 \cdots \ \stackrel{l}{\circ}.
 \end{align}
Then we have 
\begin{align*}
&v(2)+ v(3)+\cdots + v(a-2)<v(a-1)+v(a+1)+\cdots +v(l), \\
&v(2)+ v(3)+\cdots + v(a-2)+v(a-1)\ge v(a+1)+\cdots +v(l), 
\end{align*}
and $s(\Lambda)=(-1)^{l}$. 
\item Suppose that $V_{\bullet}=\{2, 3\}$, i.e. 
$\pi(\Lambda)$ is 
\begin{align}\label{case3}
 \stackrel{1}{\circ} \ \stackrel{2}{\bullet} \ \stackrel{3}{\bullet} \ 
 \stackrel{4}{\circ} \ 
 \cdots \ \stackrel{l}{\circ}.
 \end{align}
 Then we have $v(1)<v(4)+\cdots +v(l)$ and $s(\Lambda)=(-1)^{l}$. 
 \item Suppose that $V_{\bullet}=\{2, a\}$
 for $a\ge 4$, i.e. $\pi(\Lambda)$ is 
 \begin{align}\label{case4}
 \stackrel{1}{\circ} \ \stackrel{2}{\bullet} \ \stackrel{3}{\circ} 
 \ \cdots \ 
 \stackrel{a-1}{\circ} \ \stackrel{a}{\bullet} \ \stackrel{a+1}{\circ}
 \ \cdots 
 \cdots \ \stackrel{l}{\circ}.
 \end{align}
Then we have
\begin{align}\label{v1}
&v(1)+ v(3)+\cdots + v(a-2)<v(a-1)+v(a+1)+\cdots +v(l), \\
\label{v2}
&v(1)+ v(3)+\cdots + v(a-2)+v(a-1)\ge v(a+1)+\cdots +v(l), 
\end{align}
and $s(\Lambda)=(-1)^{l-1}$. 
\end{itemize}
\end{lem}
\begin{proof}
By Lemma~\ref{lem:s=0}, the sequence $\pi(\Lambda)$
is one of (\ref{case1}), (\ref{case2}), (\ref{case3}), (\ref{case4}). 
In each case, $s(\Lambda)$ is easily computed by Definition~\ref{def:s}. 
For instance, let us consider the case (\ref{case4}). 
Since $\mu(v_{a-2}) \le \mu(v_{a-1})$ and $\mu(v_{a-1})>\mu(v_{a})$, 
we have 
\begin{align}
\label{mu1}&\mu(v_1 +v_2+ \cdots + v_{a-2})<\mu(v_{a-1}+\cdots +v_{l}), \\
\label{mu2}&
\mu(v_1 +v_2+ \cdots + v_{a-2}+v_{a-1})\ge \mu(v_{a}+\cdots +v_{l}).
\end{align}
Since $v_2=v_{a}=(1, 0)$, the conditions (\ref{mu1}), (\ref{mu2})
are equivalent to (\ref{v1}), (\ref{v2}) respectively. 
Conversely if conditions (\ref{v1}), (\ref{v2})
are satisfied it is easy to check that one of (a) or (b) in 
Definition~\ref{def:s}
holds at each $1\le i\le l-1$. 
In this case, the number of $1\le i\le i-1$ in which 
$\mu(v_i) \le \mu(v_{i+1})$ holds is $l-3$, 
hence $s(\Lambda)=(-1)^{l-1}$. 
\end{proof}
\subsection{Computation of $u(\Lambda)$}
In this subsection, we compute 
$u(\Lambda)$ for a data (\ref{graph})
satisfying (\ref{rn}). 
We fix an identification between 
$V$ and $\{1, 2, \cdots, l\}$ via $\le$.
Let us take $1\le l'\le l$ and a map
\begin{align}\label{map:psi}
\psi \colon
\{1, 2, \cdots, l\} \twoheadrightarrow \{1, 2, \cdots, l'\}, 
\end{align}
which appears in (\ref{sum:u}). 
Note that $\pi(i)=\pi(j)$ if $\psi(i)=\psi(j)$, 
hence the map $\pi$
descends to the map 
\begin{align}\label{descend}
\pi' \colon 
\{1, \cdots, l'\} \to \{\bullet, \circ\},
\end{align}
via $\psi$. 
We set $v'\colon \{1, \cdots, l'\}
\to \mathbb{Z}_{\ge 1}$ to be 
\begin{align*}
v'(i)=\sum_{j\in \psi^{-1}(i)}v(j).
\end{align*}
Then the data 
\begin{align*}
\Lambda'=(\{1, \cdots, l'\}, \pi', v', \le),
\end{align*}
is a bi-colored weighted ordered vertex. 
The map $\psi$ descends to the map of the sequences 
$\pi(\psi) \colon \pi(\Lambda) \to \pi'(\Lambda')$.
First we compute the case of $\lvert V_{\bullet} \rvert =1$. 
\begin{lem}\label{lem:calu}
Suppose that $V_{\bullet}=\{a\}$ for $1\le a \le l$. 
Then we have 
\begin{align}\label{formu}
u(\Lambda)=\frac{(-1)^{l-a}}{(a-1)!(l-a)!}.
\end{align}
\end{lem}
\begin{proof}
In this case, we have 
$v(a)=2$ and the number $l''$ which appears in (\ref{sum:u})
must be $1$. 
For a map (\ref{map:psi}), 
the map $\pi(\psi) \colon 
\pi(\Lambda) \to \pi'(\Lambda')$ is either one of the following forms
by Lemma~\ref{lem:s=0}, 
\begin{align*}&a=1 \qquad
\begin{array}{cccc}
\stackrel{1}{\bullet} & \circ \cdots \circ & \cdots & \circ \cdots \circ \\
\downarrow & \downarrow & & \downarrow \\
\bullet & \circ & \cdots & \circ 
\end{array} \\
&a\ge 2 \qquad 
\begin{array}{ccccc}
\circ \cdots \circ & \stackrel{a}{\bullet} &\circ \cdots \circ
 & \cdots & \circ \cdots \circ \\
\downarrow & \downarrow &\downarrow & & \downarrow \\
\circ & \bullet & \circ & \cdots & \circ 
\end{array} 
\end{align*}
For simplicity we calculate the case of 
$a\ge 2$. The case of $a=1$ is similar. 
By the definition of $u_l$ in (\ref{sum:u})
and using Lemma~\ref{lem:coms}, we have 
$$u(\Lambda)=\frac{1}{(a-1)!}
\sum_{\begin{subarray}{c}
\psi \colon \{a+1, \cdots, l\} \to \{3, \cdots, l'\}, \\
\psi \emph{ \rm{is a non-decreasing}} \\
\emph{\rm{surjective map}}.
\end{subarray}}(-1)^{l'}
\prod_{i=1}^{l'}\frac{1}{\lvert \psi^{-1}(i) \rvert !}.$$
Then we apply Lemma~\ref{lem:elem} below 
and conclude (\ref{formu}). 
\end{proof}
We have used the following lemma, whose proof 
is written in~\cite[Proposition~4.9]{Joy4}.
\begin{lem}\label{lem:elem}
For any $l \ge 1$, we have 
$$\sum_{\begin{subarray}{c}l'\ge 0, \
\psi \colon \{1, \cdots, l\} \to \{1, \cdots, l'\}, \\
\psi \emph{ is a non-decreasing surjective map.}
\end{subarray}}(-1)^{l-l'}
\prod_{i=1}^{l'}\frac{1}{\lvert \psi^{-1}(i) \rvert !}
=\frac{1}{l!}.$$
\end{lem}
Next we compute $u(\Lambda)$ when 
$\lvert V_{\bullet} \rvert =2$. 
We write $V_{\bullet}=\{a, b\}$ for 
$1\le a<b \le l$. 
Note that we have
\begin{align*}
v(a)=v(b)=1, \quad l''\le 2.
\end{align*}
Here $l''$ is a number which appears
in (\ref{sum:u}). 
When $b-a \ge 3$, the coefficient $u(\Lambda)$
does not contribute to the sum (\ref{DTrn})
by the following lemma. 
\begin{lem}\label{lem:nocont}
Suppose that $V_{\bullet}=\{a, b\}$
with $b-a \ge 3$. Then we have 
\begin{align}\label{sum=0}
\sum_{(E, s, t)\in \eE(\Lambda)}(-1)^{\lvert E_{\bullet \to \circ} \rvert }
\prod_{e\in E}v(s(e))v(t(e))=0.
\end{align}
\end{lem}
\begin{proof}
Take $(E, s, t) \in \eE(\Lambda)$. 
Since the quiver $(V, E, s, t)$ is connected and simply 
connected, there is unique $a<i<b$
and $e, e'\in E$ such that 
\begin{align*}
s(e)=a, \ t(e)=s(e')=i, \ t(e')=b.
\end{align*}
Since $b-a\ge 3$, there is $a<j<b$ such 
that $j\neq i$. Since $(V, E, s, t)$
is connected, 
there is $e'' \in E$ such that 
either $(s(e''), t(e''))=(a, j)$ or 
$(s(e''), t(e''))=(j, b)$
holds. 
Suppose that $(s(e''), t(e''))=(a, j)$ holds, i.e. 
the geometric realization of the quiver 
$(V, E, s, t)$ is as follows,  
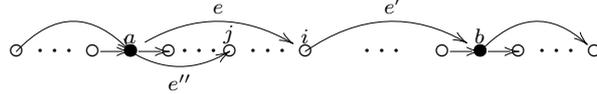
\begin{align*}
\begin{xy} 
 {\ar @/^4mm/ (22,0)*{\circ}*+!D{};(37,0) *{}
 \ar @{.}(27,0)*{\cdots}*+!D{};(27,0) *{}
 \ar  (32,0)*{\circ}*+!D{};(37,0) *{\circ}*+!D{}
 \ar @/^4mm/^{e} (37,0)*{\bullet}*+!D{_{a}};(60,0) *{\circ}*+!D{_{i}}
 \ar @/_2mm/_{e''} (37,0)*{\bullet}*+!D{};(50,0) *{\circ}*+!D{_{j}}
  \ar  (37,0)*{}*+!D{};(42,0) *{\circ}*+!D{}
  \ar @{.}(46,0)*{\cdots }*+!D{};(46,0) *{}
 \ar @{.}(55,0)*{\cdots }*+!D{};(55,0) *{}
 \ar @{.}(70,0)*{\cdots }*+!D{};(70,0) *{}
 \ar (78,0)*{\circ}*+!D{};(83,0) *{\bullet}*+!D{}
 \ar (83,0)*{}*+!D{};(88,0) *{\circ}*+!D{}
 \ar @/^4mm/^{e'} (60,0)*{}*+!D{};(83,0) *{\bullet}*+!D{_{b}}
 \ar @{.}(93,0)*{\cdots}*+!D{};(93,0) *{}
 \ar @/^4mm/ (83,0)*{}*+!D{};(98,0)*{\circ}
  }
  \end{xy}
\end{align*}
Note that by the simply connectedness of 
$(V, E, s, t)$, there is no $e'''\in E$ which satisfies 
$(s(e'''), t(e'''))=(j, b)$. We set $E'$ to be the set 
$$E'=(E\setminus\{e''\}) \coprod \{e'''\}, $$
and define maps $s', t' \colon E'\to V$ 
so that $s'|_{E\setminus \{e''\}}=s|_{E\setminus \{e''\}}$, 
$t'|_{E\setminus \{e''\}}=t|_{E\setminus \{e''\}}$, 
and $(s(e'''), t(e'''))=(j, b)$. 
The geometric realization of the 
quiver $(V, E', s', t')$ is as follows,  
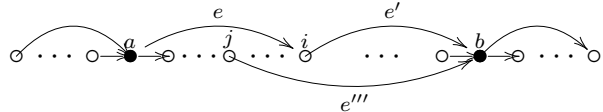
\begin{align*}
\begin{xy} 
 {\ar @/^4mm/ (22,0)*{\circ}*+!D{};(37,0) *{}
 \ar @{.}(27,0)*{\cdots}*+!D{};(27,0) *{}
 \ar  (32,0)*{\circ}*+!D{};(37,0) *{\circ}*+!D{}
 \ar @/^4mm/^{e} (37,0)*{\bullet}*+!D{_{a}};(60,0) *{\circ}*+!D{_{i}}
  \ar  (37,0)*{}*+!D{};(42,0) *{\circ}*+!D{}
  \ar @{.}(46,0)*{\cdots }*+!D{};(46,0) *{}
 \ar @{.}(55,0)*{\cdots }*+!D{};(55,0) *{}
 \ar @{.}(70,0)*{\cdots }*+!D{};(70,0) *{}
 \ar (78,0)*{\circ}*+!D{};(83,0) *{\bullet}*+!D{}
 \ar (83,0)*{}*+!D{};(88,0) *{\circ}*+!D{}
 \ar @/^4mm/^{e'} (60,0)*{}*+!D{};(83,0) *{\bullet}*+!D{_{b}}
 \ar @{.}(93,0)*{\cdots}*+!D{};(93,0) *{}
 \ar @/^4mm/ (83,0)*{}*+!D{};(98,0)*{\circ}
 \ar @/_4mm/ _{e'''} (50,0)*{\circ}*+!D{_{j}};(83,0) *{\bullet}
  }
  \end{xy}
\end{align*}
Since $v(a)=v(b)=1$, we have 
\begin{align*}
(-1)^{\lvert E_{\bullet \to \circ} \rvert }
\prod_{e\in E}v(s(e))v(t(e)) +
(-1)^{\lvert E'_{\bullet \to \circ} \rvert }
\prod_{e\in E'}v(s'(e))v(t'(e))
=0.
\end{align*}
Therefore the sum (\ref{sum=0}) vanishes. 
\end{proof}
We compute $u(\Lambda)$ when 
$b-a \le 2$. 
Let us divide $u(\Lambda)$ into the following sum, 
\begin{align*}
u(\Lambda)=u^{(1)}(\Lambda)+u^{(2)}(\Lambda)+u^{(3)}(\Lambda).
\end{align*}
Each $u^{(i)}(\Lambda)$ is the following. 
\begin{itemize}
\item $u^{(1)}(\Lambda)$ is defined by the sum (\ref{sum:u})
with $l''=1$ and
$\psi \colon \{1, \cdots, l\} \to \{1, \cdots, l'\}$ 
satisfying $\lvert \pi^{'-1}(\bullet) \rvert =2$. 
Here $\pi'\colon \{1, \cdots l'\} \to \{\bullet, \circ\}$
is given by (\ref{descend}). 
\item $u^{(2)}(\Lambda)$ is defined by the sum (\ref{sum:u})
with $l''=1$ and
$\psi \colon \{1, \cdots, l\} \to \{1, \cdots, l'\}$ 
satisfying $\lvert \pi^{'-1}(\bullet) \rvert =1$. 
\item $u^{(3)}(\Lambda)$ is defined by the sum (\ref{sum:u})
with $l''=2$. 
\end{itemize}
We compute $u^{(1)}(\Lambda)$ as follows. 
\begin{lem}\label{lem:u(1)}
(i) Suppose that $V_{\bullet}=\{a, a+1\}$
for $1\le a \le l-1$. Then $u^{(1)}(\Lambda)$is 
non-zero if and only if 
\begin{align*}
v(1)+\cdots +v(a-1)<v(a+2)+\cdots +v(l). 
\end{align*}
In this case, we have 
\begin{align*}
u^{(1)}(\Lambda)=\frac{(-1)^{l-a}}{(a-1)!(l-a-1)!}.
\end{align*}
(ii) Suppose that $V_{\bullet}=\{a, a+2\}$
for $1\le a \le l-2$. 
 Then $u^{(1)}(\Lambda)$ is non-zero if and only if 
 \begin{align}\label{v<}
 & v(1)+\cdots +v(a-1)<v(a+1)+v(a+3)+\cdots +v(l), \\ \label{vge}
 & v(1)+\cdots +v(a-1)+v(a+1)\ge v(a+3)+\cdots +v(l).
 \end{align}
 In this case, we have 
\begin{align}\label{u(1)}
u^{(1)}(\Lambda)=\frac{(-1)^{l-a-1}}{(a-1)!(l-a-2)!}.
\end{align}
\end{lem}
\begin{proof}
The computations of (i) and (ii) are identical, so 
we only check (ii). 
Let $\psi \colon \{1, \cdots, l\} \to \{1, \cdots, l'\}$
be a map which appears in (\ref{sum:u}). 
By Lemma~\ref{lem:s=0}, the map 
$\pi(\psi)\colon \pi(\Lambda) \to \pi'(\Lambda')$ is one 
of the following forms, 
\begin{align*}&a=1 \qquad
\begin{array}{cccccc}
\stackrel{1}{\bullet} & \circ & \stackrel{3}{\bullet}
 & \circ \cdots \circ & \cdots & \circ \cdots \circ
\\
\downarrow & \downarrow & \downarrow & \downarrow & & \downarrow \\
\bullet & \circ & \bullet & \circ & \cdots & \circ 
\end{array} \\
&a\ge 2 \qquad 
\begin{array}{ccccccc}
\circ \cdots \circ & 
\stackrel{a}{\bullet} & \circ & \stackrel{a+2}{\bullet}
 & \circ \cdots \circ & \cdots & \circ \cdots \circ
\\
\downarrow & 
\downarrow & \downarrow & \downarrow & \downarrow & & \downarrow \\
\circ & \bullet & \circ & \bullet & \circ & \cdots & \circ 
\end{array}
\end{align*}
For simplicity we calculate the case of $a\ge 2$. 
By Lemma~\ref{lem:computed}, we see that $u^{(1)}(\Lambda)$ 
is non-zero only if (\ref{v<}) and (\ref{vge}) hold. 
By Lemma~\ref{lem:computed} and (\ref{sum:u}), we have 
\begin{align*}
u^{(1)}(\Lambda)=\frac{1}{(a-1)!}
\sum_{\begin{subarray}{c}
\psi \colon \{a+3, \cdots, l\} \to \{5, \cdots, l'\}, \\
\psi \emph{ \rm{is a non-decreasing} }\\
\emph{\rm{surjective map.}}
\end{subarray}}(-1)^{l'-1}
\prod_{i=5}^{l'}\frac{1}{\lvert \psi^{-1}(i)\rvert !}.
\end{align*}
Applying Lemma~\ref{lem:elem}, we obtain (\ref{u(1)}). 
\end{proof}
The computation of $u^{(2)}(\Lambda)$ is as follows. 
\begin{lem}\label{lem:compu2}
We have $u^{(2)}(\Lambda)\neq 0$ if and only if 
$V_{\bullet}=\{a, a+1\}$ for some $1\le a \le l-1$. 
In this case, we have 
\begin{align}\label{u2}
u^{(2)}(\Lambda)=\frac{(-1)^{l-a-1}}{2(a-1)!(l-a-1)!}.
\end{align}
\end{lem}
\begin{proof}
Suppose that $u^{(2)}(\Lambda)\neq 0$. 
By the definition of 
$u^{(2)}(\Lambda)$, 
it is obvious that $V_{\bullet}=\{a, a+1\}$
for some $1\le a \le l-1$. 
By Lemma~\ref{lem:s=0}, 
the map $\pi(\psi) \colon \pi(\Lambda) \to \pi'(\Lambda')$
is one of the following forms, 
\begin{align*}&a=1 \qquad
\begin{array}{cccc}
\stackrel{1}{\bullet} \ 
\stackrel{2}{\bullet} & \circ \cdots \circ & \cdots & \circ \cdots \circ
\\
\downarrow  & \downarrow & & \downarrow \\
\bullet  & \circ & \cdots & \circ 
\end{array} \\
&a\ge 2 \qquad 
\begin{array}{ccccc}
\circ \cdots \circ & 
\stackrel{a}{\bullet} \ 
\stackrel{a+1}{\bullet}  & \circ \cdots \circ & \cdots & \circ \cdots \circ
\\
\downarrow & 
\downarrow  & \downarrow & & \downarrow \\
\circ & \bullet & \circ & \cdots & \circ 
\end{array}
\end{align*}
For simplicity we treat the case of $a\ge 2$. 
By Lemma~\ref{lem:coms} and the definition of $u^{(2)}(\Lambda)$, we have 
\begin{align*}
u^{(2)}(\Lambda)=\frac{1}{2(a-1)!}
\sum_{\begin{subarray}{c}
\psi \colon \{a+2, \cdots, l\} \to \{3, \cdots, l'\}, \\
\psi \emph{ \rm{is a non-decreasing}} \\
\emph{\rm{surjective map}}.
\end{subarray}}(-1)^{l'}
\prod_{i=3}^{l'}\frac{1}{\lvert \psi^{-1}(i)\rvert !}.
\end{align*}
Applying Lemma~\ref{lem:elem}, we obtain (\ref{u2}). 
\end{proof}
Finally we compute $u^{(3)}(\Lambda)$. 
\begin{lem}\label{lem:compu3}
(i) Suppose that $V_{\bullet}=\{a, a+1\}$
for $1\le a \le l-1$. 
Then $u^{(3)}(\Lambda)$ is non-zero if and only if 
the following condition holds, 
\begin{align*}
v(1)+v(2)+\cdots +v(a-1)=v(a+2)+\cdots +v(l).
\end{align*}
In this case, we have 
\begin{align*}
u^{(3)}(\Lambda)=\frac{(-1)^{l-a}}{2(a-1)!(l-a-1)!}.
\end{align*}

(ii) Suppose that $V_{\bullet}=\{a, a+2\}$
for $1\le a \le l-2$. Then $u^{(3)}(\Lambda)$ is non-zero 
either one of the following conditions holds, 
\begin{align}\label{sumv1}
&v(1)+\cdots +v(a-1)=v(a+1)+\cdots +v(l), \\
\label{sumv2}
&v(1)+\cdots +v(a-1)+v(a+1)=v(a+2)+\cdots +v(l).
\end{align}
If (\ref{sumv1}) (resp.~(\ref{sumv2})) holds, 
then we have 
\begin{align}\label{comu2}
u^{(3)}(\Lambda)=\frac{(-1)^{l-a-1}}{2(a-1)!(l-a-1)!}, 
\quad \left(resp.~\frac{(-1)^{l-a}}{2(a-1)!(l-a-1)!}.\right)
\end{align}
\end{lem}
\begin{proof}
The computations of (i), (ii) are identical, 
so we only check (ii). Suppose that $u^{(3)}(\Lambda)\neq 0$
and let 
$\psi \colon \{1, \cdots, l\} \to \{1, \cdots, l'\}$ and 
$\xi \colon \{1, \cdots, l'\} \to \{1, 2\}$
be maps which appear in the sum (\ref{sum:u}). 
By the definition of $u^{(3)}(\Lambda)$, 
the subset $(\psi \circ \xi)^{-1}(i)$
contains an element of $V_{\bullet}$
for $i=1, 2$. Therefore 
$(\psi \circ \xi)^{-1}(1)$ is one of the following, 
\begin{align}\label{one1}
&(\psi \circ \xi)^{-1}(1)=\{1, 2, \cdots, a\}, \\
\label{one2}
&(\psi \circ \xi)^{-1}(1)=\{1, 2, \cdots, a, a+1\}.
\end{align}
If (\ref{one1}) (resp.~(\ref{one2})) holds, 
then the condition (\ref{sumv1}) (resp.~(\ref{sumv2})) holds. 
For simplicity we treat the case in which (\ref{one1}) holds. 
The map $\pi(\psi)\colon 
\pi(\Lambda) \to \pi'(\Lambda')$
together with the map $\xi$ is as follows, 
\begin{align*}
\begin{array}{ccccccc}
(\circ \cdots \circ & 
\stackrel{a}{\bullet}) & (\circ & \stackrel{a+2}{\bullet}
 & \circ \cdots \circ & \cdots & \circ \cdots \circ)
\\
\downarrow & 
\downarrow & \downarrow & \downarrow & \downarrow & & \downarrow \\
(\circ & \bullet) & (\circ & \bullet & \circ & \cdots & \circ ) \\
\quad \qquad \downarrow {_{\xi}} & & & &  \downarrow {_{\xi}} & & \\
\quad \qquad 1 & & & &2 &  & 
\end{array}
\end{align*}
By Lemma~\ref{lem:coms} and the definition of $u^{(3)}(\Lambda)$, 
we have 
\begin{align*}
u^{(3)}(\Lambda)=-\frac{1}{2}
\cdot \frac{1}{(a-1)!}
\sum_{\begin{subarray}{c}
\psi \colon \{a+3, \cdots, l\} \to \{5, \cdots, l'\}, \\
\psi \emph{ \rm{is a non-decreasing}} \\
\emph{\rm{surjective map.}}
\end{subarray}}(-1)^{l'}
\prod_{i=5}^{l'}\frac{1}{\lvert \psi^{-1}(i)\rvert !}.
\end{align*}
Applying Lemma~\ref{lem:elem}, we obtain (\ref{comu2}). 
\end{proof}

\subsection{Generating series of $\DT(2, n)$}
Combining the calculations in the previous subsections, 
we compute $\DT(2, n)$. 
We divide $\DT(2, n)$ into the 
following four parts, 
\begin{align*}
\DT(2, n)=\DT^{(0)}(2, n)+\DT^{(1)}(2, n) +\DT^{(2)}(2, n)
+\DT^{(3)}(2, n).
\end{align*}
Each $\DT^{(i)}(2, n)$ is the following. 
\begin{itemize}
\item $\DT^{(0)}(2, n)$ is defined by the sum (\ref{DTrn})
for bi-colored weighted ordered vertices
$\Lambda=(V, \pi, v, \le)$ with $r(\Lambda)=2$, $n(\Lambda)=n$
and $\lvert V_{\bullet}\rvert =1$. 
\item For $1\le i\le 3$, 
$\DT^{(i)}(2, n)$ is defined by the sum (\ref{DTrn})
for bi-colored weighted ordered vertices $\Lambda=(V, \pi, v, \le)$
with $r(\Lambda)=2$, $n(\Lambda)=n$,
$\lvert V_{\bullet}\rvert =2$, 
and $u(\Lambda)$ is replaced by $u^{(i)}(\Lambda)$. 
\end{itemize}
We define the generating series $\DT^{(i)}(2)$ by 
\begin{align*}
\DT^{(i)}(2)=\sum_{n\ge 0}\DT^{(i)}(2, n)q^n.
\end{align*}
In what follows, we compute $\DT^{(i)}(2)$. 
Recall the definition of the MacMahon function $M(q)$ 
given in (\ref{Mac}). 
\begin{lem}\label{lem:DT^0}
We have the following formula.
\begin{align}\label{DT(0)}
\DT^{(0)}(2)=\frac{1}{4}M(q)^{2\chi(X)}.
\end{align}
\end{lem}
\begin{proof}
Let $\Lambda=(V, \pi, v, \le)$ be a  
bi-colored weighted ordered vertex with 
$\lvert V \rvert =l$ and $V_{\bullet}=\{a\}$
for $1\le a \le l$. 
Obviously the set $\eE(\Lambda)$ consists of 
one element $(E, s, t)\in \eE(\Lambda)$,
 whose geometric realization is as 
follows, 
\begin{align*}
\begin{xy} 
 { \ar @/^4mm/ (37,0)*{\circ}*+!D{_{1}};(60,0) *{\bullet}*+!D{_{a}}
 \ar @{.}(48,0)*{\cdots }*+!D{};(48,0) *{}
 \ar  (55,0)*{\circ}*+!D{};(60,0) *{}
 \ar (60,0)*{}*+!D{};(65,0) *{\circ}
 \ar @{.}(74,0)*{\cdots }*+!D{};(74,0) *{}
 \ar @/^4mm/ (60,0)*{}*+!D{};(83,0) *{\circ}*+!D{_{l}}
   }
  \end{xy}
\end{align*}
Note that we have $\lvert E_{\bullet \to \circ} \rvert =l-a$. 
By Remark~\ref{rmk:Mac}, Theorem~\ref{thm:DTrn}
 and Lemma~\ref{lem:calu}, we have 
 \begin{align}\notag
 \DT^{(0)}(2)&=\sum_{\begin{subarray}{c}
 l\ge 1, \ 1\le a \le l, \\
 v\colon \{1, \cdots, l\} \to \mathbb{Z}_{\ge 1}, \\
 v(a)=2. \end{subarray} }
 \frac{(-1)^{l-a}}{(a-1)!(l-a)!}\cdot \frac{1}{4}\prod_{i\neq a}
 \DT(0, v(i))q^{v(i)} \\ 
 \notag
 & \qquad \qquad \qquad \qquad \qquad
 \qquad \left( -\frac{1}{2}\right)^{l-1}\cdot 
 (-1)^{l-a}\prod_{i\neq a}2v(i) \\
 \label{comDT0}
 &=\frac{1}{4}\sum_{\begin{subarray}{c}
 l\ge 0,  \\
 v\colon \{1, \cdots, l\} \to \mathbb{Z}_{\ge 1}.
 \end{subarray}}
 \frac{1}{l!}\prod_{i=1}^{l}(-2v(i))\DT(0, v(i))q^{v(i)} \\
 \notag
 &=\frac{1}{4}\exp\left(\sum_{n\ge 0}-2n \DT(0, n)q^n \right) \\
 \label{comDT00}
 &=\frac{1}{4}M(q)^{2\chi(X)}.
 \end{align}
 Here we have used the following in (\ref{comDT0}), 
 \begin{align*}
 \sum_{1\le a\le l}\frac{1}{(a-1)!(l-a)!}\cdot \frac{1}{2^{l-1}}
 =\frac{1}{(l-1)!},
 \end{align*}
and the formula (\ref{form:exp}) in (\ref{comDT00}). 
\end{proof}
Next let us compute $\DT^{(1)}(2)$. 
We introduce the following notation. 
We define the series $N(q)$ to be 
\begin{align*}N(q) 
&\cneq q\frac{d}{dq}\log M(q) \\
&=\sum_{n\ge 0, \ r|n}r^2 q^n. 
\end{align*}
For series $f_1, f_2, \cdots, f_N \in \mathbb{Q}\db[q\db]$
given by 
$$f_i=\sum_{n\ge 0}a_{n}^{(i)}q^n, \quad 1\le i\le N, $$
and a subset $\Delta \subset \mathbb{Z}^{N}_{\ge 0}$, we define the series 
$\{f_1 \cdot f_2 \cdots f_N\}_{\Delta}$ to be 
\begin{align}\label{nota}
\{f_1 \cdot f_2 \cdots f_N\}_{\Delta}
=\sum_{(n_1, n_2, \cdots, n_N)\in \Delta}
a_{n_1}^{(1)}a_{n_2}^{(2)}\cdots a_{n_N}^{(N)}
q^{n_1+n_2+\cdots +n_N}.
\end{align}
\begin{lem}\label{lem:formDT22}
We have the following formula,
\begin{align}\label{form:DT22}
\DT^{(1)}(2)=
-\frac{\chi(X)}{2}\{M(q)^{\chi(X)}
\cdot M(q)^{\chi(X)}\cdot N(q)\}_{\Delta}.
\end{align}
Here $\Delta \subset \mathbb{Z}^3_{\ge 0}$ is 
\begin{align}\label{lem:Delta}
\Delta=\{(m_1, m_2, m_3)\in \mathbb{Z}^3_{\ge 0} : 
-m_3 \le m_1-m_2 < m_3\}.
\end{align}
\end{lem}
\begin{proof}
Let $\Lambda=(V, \pi, v, \le)$ be 
a bi-colored weighted ordered vertex with 
$r(\Lambda)=2$, $n(\Lambda)=n$ and 
$\lvert V_{\bullet}\rvert =2$.
Let $\lvert V \rvert =l$ and we 
identify $V$ and $\{1, \cdots, l\}$
via $\le$. 
By Lemma~\ref{lem:nocont}, the data $\Lambda$
contributes to (\ref{DTrn}) 
only if one of the following conditions
hold. 
\begin{itemize}
\item We have $V_{\bullet}=\{a, a+1\}$ for 
$1\le a\le l-1$. 
In this case, there are two types for $(E, s, t)\in \eE(\Lambda)$. 

{\bf Type A: }
There is unique $1\le c\le a-1$ and $e, e'\in E$
such that 
\begin{align*}
s(e)=s(e')=c, \quad t(e)=a, \quad t(e')=a+1. 
\end{align*}
In this case, we have $\lvert E_{\bullet \to \circ} \rvert=l-a-1$. 
If we fix such $c$, there are $2^{l-3}$-choices of such 
$(E, s, t)\in \eE(\Lambda)$. One of their geometric realizations 
is as follows, 
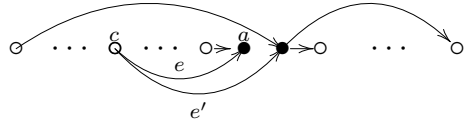
\begin{align*}
\begin{xy} 
 { \ar @/^6mm/ (30,0)*{\circ}*+!D{};(65,0) *{}
 \ar @/_4mm/^{e} (43,0)*{\circ}*+!D{_{c}};(60,0) *{}
  \ar @/_6mm/_{e'} (43,0)*{\circ}*+!D{};(65,0) *{}
 \ar @{.}(37,0)*{\cdots }*+!D{};(37,0) *{}
  \ar @{.}(49,0)*{\cdots }*+!D{};(49,0) *{}
 \ar  (55,0)*{\circ}*+!D{};(60,0) *{\bullet}*+!D{_{a}}
 \ar (65,0)*{\bullet}*+!D{};(70,0) *{\circ}
 \ar @{.}(79,0)*{\cdots }*+!D{};(79,0) *{}
 \ar @/^6mm/ (65,0)*{}*+!D{};(88,0) *{\circ}
   }
  \end{xy}
\end{align*}

{\bf Type B: }
There is unique $a+2 \le c\le l$ and $e, e'\in E$
such that 
\begin{align*}
t(e)=t(e')=c, \quad s(e)=a+1, \quad s(e')=a. 
\end{align*}
In this case, we have $\lvert E_{\bullet \to \circ} \rvert=l-a$. 
If we fix such $c$, there are $2^{l-3}$-choices of such 
$(E, s, t)\in \eE(\Lambda)$. One of their geometric realizations 
is as follows, 
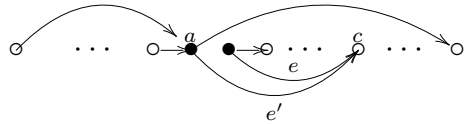
\begin{align*}
\begin{xy} 
 { \ar @/^6mm/ (30,0)*{\circ}*+!D{};(53,0) *{\bullet}*+!D{_{a}}
  \ar (48,0)*{\circ}*+!D{};(53,0) *{}*+!D{}
 \ar @/_4mm/^{e} (58,0)*{\bullet}*+!D{};(75,0) *{}*+!D{}
   \ar (58,0)*{}*+!D{};(63,0) *{\circ}*+!D{}
 \ar @/_6mm/_{e'} (53,0)*{}*+!D{};(75,0) *{\circ}*+!D{_{c}}
 \ar @/^6mm/ (53,0)*{}*+!D{};(88,0) *{\circ}*+!D{}
 \ar @{.}(40,0)*{\cdots }*+!D{};(40,0) *{}
 \ar @{.}(68,0)*{\cdots }*+!D{};(68,0) *{}
 \ar @{.}(81,0)*{\cdots }*+!D{};(81,0) *{}
     }
  \end{xy}
\end{align*}
\item 
We have $V_{\bullet}=\{a, a+2 \}$
for $1\le a \le l-2$. In this case, 
we call an element $(E, s, t)\in \eE(\Lambda)$ as Type C. 

{\bf Type C:}
There is $e, e'\in E$ 
such that 
\begin{align*}
s(e)=a, \quad t(e)=s(e')=a+1, \quad t(e')=a+2.
\end{align*}
There are $2^{l-3}$-choices of $(E, s, t)\in \eE(\Lambda)$. 
One of their geometric realizations is as follows,
\begin{align*}
\begin{xy} 
 { \ar @/_6mm/ (30,0)*{\circ}*+!D{};(63,0) *{\bullet}*+!D{}
  \ar (48,0)*{\circ}*+!D{};(53,0) *{\bullet}*+!D{_{a}}
  \ar_{e} (53,0)*{}*+!D{};(58,0) *{\circ}*+!D{}
    \ar^{e'} (58,0)*{}*+!D{};(63,0) *{\bullet}*+!D{}
    \ar (63,0)*{}*+!D{};(68,0) *{\circ}*+!D{}
 \ar @/^6mm/ (53,0)*{}*+!D{};(88,0) *{\circ}*+!D{}
 \ar @{.}(40,0)*{\cdots }*+!D{};(40,0) *{}
  \ar @{.}(78,0)*{\cdots }*+!D{};(78,0) *{}
     }
  \end{xy}
  \end{align*}
\end{itemize}
We write $\DT^{(1)}(2)$ as 
\begin{align*}
\DT^{(1)}(2)=\DT_{A}^{(1)}(2)+\DT_{B}^{(1)}(2)+\DT_{C}^{(1)}(2), 
\end{align*}
where $\DT_{A}^{(1)}(2)$, $\DT_{B}^{(1)}(2)$
and $\DT_{C}^{(1)}(2)$ are contributions of $(E, s, t)\in \eE(\Lambda)$
of type $A$, $B$ and $C$ respectively. 
Using Lemma~\ref{lem:u(1)} (i) and Theorem~\ref{thm:DTrn}, 
the series $\DT_{A}^{(1)}(2)$
is computed as follows, 
\begin{align}\notag
\DT_{A}^{(1)}(2)&=\sum_{\begin{subarray}{c}
l\ge 1, \ 1\le a \le l-1, \ 1\le c\le a-1, \\
\notag
v\colon \{1, \cdots, l\} \to \mathbb{Z}_{\ge 1}, \
v(a)=v(a+1)=1, \\
\notag
v(1)+\cdots +v(a-1)<v(a+2)+\cdots +v(l).
\end{subarray}}
\frac{(-1)^{l-a}}{(a-1)!(l-a-1)!}\prod_{i\neq a, a+1}
\DT(0, v(i))q^{v(i)} \\
\notag
&\qquad \qquad \qquad \qquad \qquad \qquad 
\left( -\frac{1}{2} \right)^{l-1}\cdot (-1)^{l-a-1} \cdot 2^{l-3}
\prod_{i\neq c}v(i) \cdot v(c)^2 \\
\notag
&=\frac{1}{4}\sum_{\begin{subarray}{c}
a\ge 0, \ b\ge 0, \ k\ge 1, \\
v\colon \{1, \cdots, a\} \to \mathbb{Z}_{\ge 1}, \
v' \colon \{1, \cdots, b\} \to \mathbb{Z}_{\ge 1}, \\
v(1)+\cdots +v(a)+k<v'(1)+\cdots +v'(b).
\end{subarray}}
\frac{1}{a!}\prod_{i=1}^{a}(-v(i))\DT(0, v(i))q^{v(i)} \\
\notag
& \qquad \qquad \qquad \qquad \qquad 
\cdot \frac{1}{b!}\prod_{i=1}^{b}(-v'(i))\DT(0, v(i))q^{v'(i)} \cdot 
(-k^2)\DT(0, k)q^k \\
\label{DT_A}
&=\frac{\chi(X)}{4}\{M(q)^{\chi(X)} \cdot M(q)^{\chi(X)} \cdot 
N(q)\}_{\Delta_A}.
\end{align}
 Here $\Delta_{A}$ is defined by 
 \begin{align*}
 \Delta_{A}=\{(m_1, m_2, m_3)\in \mathbb{Z}_{\ge 0}^{3} :
 m_1 +m_3<m_2\},
 \end{align*}
 and we have used the formula (\ref{DT0n}) 
 in (\ref{DT_A}). 
 Using Lemma~\ref{lem:u(1)}, similar computations show 
 that 
 \begin{align*}
 \DT_{B}^{(1)}(2)&=
 -\frac{\chi(X)}{4}\{M(q)^{\chi(X)} \cdot M(q)^{\chi(X)} \cdot 
N(q)\}_{\Delta_B}, \\
\DT_{C}^{(1)}(2)&=
 -\frac{\chi(X)}{4}\{M(q)^{\chi(X)} \cdot M(q)^{\chi(X)} \cdot 
N(q)\}_{\Delta},
 \end{align*}
 where $\Delta_B$ is defined by 
 \begin{align*}
 \Delta_{B}&=\{(m_1, m_2, m_3)\in \mathbb{Z}_{\ge 0}^{3} :
 m_1 <m_2 +m_3\}, 
 \end{align*}
 and $\Delta$ is defined by (\ref{lem:Delta}). 
 Noting that 
 $$\Delta_{B}=\Delta_{A} \coprod \Delta,$$
 we obtain the formula (\ref{form:DT22}). 
 \end{proof}
 Finally we show that 
 $\DT^{(i)}(2)$ vanish for $i=2, 3$. 
 \begin{lem}\label{lem:DT^2}
 We have $\DT^{(i)}(2, n)=0$ for any $n\ge 0$ and 
 $i=2, 3$. 
 \end{lem}
 \begin{proof}
 Let $\Lambda=(V, \pi, v, \le)$ 
 be a bi-colored weighted ordered vertex with 
 $r(\Lambda)=2$, $\lvert V \rvert =l$, 
 and take $(E, s, t)\in \eE(\Lambda)$. 
 By Lemma~\ref{lem:nocont}, we may assume that $V_{\bullet}=\{a, a+1\}$
 or $V_{\bullet}=\{a, a+2\}$
 for some $1\le a \le l-1$. 
 Let us consider the following data, 
 \begin{align*}
 \Lambda^{\ast}=(V, \pi, v, \le^{\ast}), \quad 
 (E, s^{\ast}, t^{\ast}),
  \end{align*}
  by setting $\le^{\ast}$, $s^{\ast}$ and $t^{\ast}$
  to be 
  \begin{align*}
  \alpha \le^{\ast}\beta
  \mbox{ if and only if } \alpha \ge \beta, 
  \quad s^{\ast}=t, \quad 
  t^{\ast}=s. 
  \end{align*}
 Then it is obvious that $(E, s^{\ast}, t^{\ast})\in 
 \eE(\Lambda^{\ast})$.
 For instance, the relationship between 
 geometric realizations is as follows,  
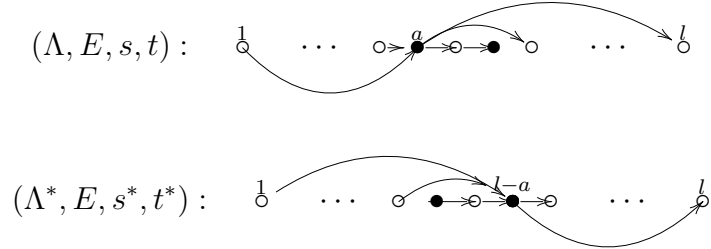
\begin{align*}
(\Lambda, E, s, t): \quad
\begin{xy} 
 { \ar @/_6mm/ (30,0)*{\circ}*+!D{_{1}};(53,0) *{\bullet}*+!D{}
 \ar (53,0)*{}*+!D{};(58,0) *{\circ}*+!D{}
  \ar (48,0)*{\circ}*+!D{};(53,0) *{\bullet}*+!D{_{a}}
  \ar @/^3mm/(53,0)*{}*+!D{};(68,0) *{\circ}*+!D{}
    \ar (58,0)*{}*+!D{};(63,0) *{\bullet}*+!D{}
     \ar @/^6mm/ (53,0)*{}*+!D{};(88,0) *{\circ}*+!D{_{l}}
 \ar @{.}(40,0)*{\cdots }*+!D{};(40,0) *{}
  \ar @{.}(78,0)*{\cdots }*+!D{};(78,0) *{}
     }
  \end{xy}
\end{align*}
\begin{align*}
(\Lambda^{\ast}, E, s^{\ast}, t^{\ast}): \quad 
\begin{xy} 
 { \ar @/^6mm/ (30,0)*{\circ}*+!D{_{1}};(63,0) *{\bullet}*+!D{}
 \ar (53,0)*{\bullet}*-!D{};(58,0) *{\circ}*+!D{}
  \ar @/^3mm/(48,0)*{\circ}*+!D{};(63,0) *{\bullet}*+!D{_{l-a}}
  \ar (63,0)*{\bullet}*+!D{};(68,0) *{\circ}*+!D{}
    \ar (58,0)*{}*+!D{};(63,0) *{\bullet}*+!D{}
     \ar @/_6mm/ (63,0)*{}*+!D{};(88,0) *{\circ}*+!D{_{l}}
 \ar @{.}(40,0)*{\cdots }*+!D{};(40,0) *{}
  \ar @{.}(78,0)*{\cdots }*+!D{};(78,0) *{}
     }
  \end{xy}
\end{align*} 
 Note that if $V_{\bullet}=\{a, a+1\}$, 
 then $(E, s, t)$ is of type A (resp.~B) in the proof of 
 Lemma~\ref{lem:formDT22} if and only if 
 $(E^{\ast}, s^{\ast}, t^{\ast})$
  is of type B (resp.~A). 
 Also if $V_{\bullet}=\{a, a+2\}$, then 
 $\Lambda$ satisfies (\ref{sumv1}), (resp.~(\ref{sumv2}))
 if and only if $\Lambda^{\ast}$
 satisfies (\ref{sumv2}), (resp.~(\ref{sumv1}).)
 Hence the map  
 $$(\Lambda, (E, s, t)) \mapsto (\Lambda^{\ast}, (E, s^{\ast}, 
 t^{\ast})),$$
 is a free involution
 on the set of pairs $(\Lambda, (E, s, t))$
 for data (\ref{graph}) satisfying $V_{\bullet}=\{a, b\}$ with 
 $0<b-a \le 2$ 
 and $(E, s, t)\in \eE(\Lambda)$. 
 Using the computations of $u^{(2)}(\Lambda)$, $u^{(3)}(\Lambda)$
 in Lemma~\ref{lem:compu2} and Lemma~\ref{lem:compu3}, it is easy to 
 check that 
 \begin{align*}
 (-1)^{\lvert E_{\bullet \to \circ} \rvert}u^{(i)}(\Lambda)
 + (-1)^{\lvert E^{\ast}_{\bullet \to \circ} \rvert}u^{(i)}
 (\Lambda^{\ast})=0,
 \end{align*} 
 for $i=2, 3$. 
 Therefore $\DT^{(i)}(2, n)=0$ for any $n\ge 0$ and 
 $i=2, 3$. 
 \end{proof}
 Summarizing Lemma~\ref{lem:DT^0}, Lemma~\ref{lem:formDT22} 
 and Lemma~\ref{lem:DT^2}, 
 we obtain the following. 
 \begin{thm}\label{thm:formDT2}
We have the following formula. 
\begin{align}\label{DT2}
\DT(2)=\frac{1}{4}M(q)^{2\chi(X)} -\frac{\chi(X)}{2}\{M(q)^{\chi(X)}
\cdot M(q)^{\chi(X)}\cdot N(q)\}_{\Delta},
\end{align}
for $\Delta=\{(m_1, m_2, m_3)\in \mathbb{Z}_{\ge 0}^3 : 
-m_3 \le m_1-m_2 < m_3\}$. 
\end{thm}
\begin{rmk}
By Corollary~\ref{cor:DT=Eu} and Theorem~\ref{thm:formDT2}, 
we have 
\begin{align*}
\sum_{n \ge 0}\Eu(2, n)q^n=
-\frac{1}{4}M(q)^{2\chi(X)} +\frac{\chi(X)}{2}\{M(q)^{\chi(X)}
\cdot M(q)^{\chi(X)}\cdot N(q)\}_{\Delta},
\end{align*}
for $\Delta=\{(m_1, m_2, m_3)\in \mathbb{Z}_{\ge 0}^3 : 
-m_3 \le m_1-m_2 < m_3\}$. 
\end{rmk}

\section{Integrality property}\label{sec:int}
In this section, we study the invariant 
$\Omega(2, n)\in \mathbb{Q}$, 
defined as follows. 
\begin{defi}\emph{
We define $\Omega(2, n)\in \mathbb{Q}$ to be 
\begin{align*}
\Omega(2, n)=\left\{ \begin{array}{cl}\DT(2, n), & n \mbox{ is odd,} \\
\DT(2, n)-\frac{1}{4}\DT(1, \frac{n}{2}), & n\mbox{ is even.}
\end{array}\right.\end{align*}
}
\end{defi}
By Corollary~\ref{cor:DT=Eu}, 
the invariant $\Omega(2, n)$ is also 
written as 
\begin{align}\label{OmegaE}
\Omega(2, n)=\left\{ \begin{array}{cl}-\Eu(2, n), & n \mbox{ is odd,} \\
-\Eu(2, n)-\frac{(-1)^{\frac{n}{2}}}{4}\Eu(1, \frac{n}{2}), & n\mbox{ is even.}
\end{array}\right.\end{align}
In this section, 
we show the following result, which 
is an evidence of the integrality 
conjecture by Kontsevich-Soibelman~\cite[Conjecture~6]{K-S}. 
\begin{thm}\label{thm:intconj}
We have $\Omega(2, n)\in \mathbb{Z}$. 
\end{thm}
It seems that Theorem~\ref{thm:intconj} is not 
obvious from the explicit formula (\ref{DT2}). 
Instead of using (\ref{DT2}), 
we give a geometric proof of Theorem~\ref{thm:intconj}
using the definition of $\DT(2, n)$. 

Let $Q^{(2, n)}\subset \Quot^{(n)}(\oO_X^{\oplus 2})$
be a $\GL(2, \mathbb{C})$-invariant 
Zariski open subset constructed in 
Lemma~\ref{Qrn}.
By Lemma~\ref{Qrn}, there is a smooth morphism
\begin{align*}
f\colon Q^{(2, n)} \to \oO bj^{(2, n)}(\aA_X).
\end{align*}
For $p\in Q^{(2, n)}$, we denote by 
$E_p \in \aA_X$ the object corresponding to $f(p)\in 
\oO bj^{(2, n)}(\aA_X)$. 

By the definition of $\DT(2, n)$, 
it is obvious that $\Omega(2, n)\in \mathbb{Z}$
when $n$ is odd. Therefore 
in what follows we set $n=2m$ for 
$m\in \mathbb{Z}$. 
We take a $\GL(2, \mathbb{C})$-invariant 
stratification of $Q^{(2, 2m)}$, 
\begin{align*}
Q^{(2, 2m)}=Q^{(2, 2m)}_{0} \coprod Q^{(2, 2m)}_{1} \coprod 
Q^{(2, 2m)}_{2} \coprod Q^{(2, 2m)}_{3} \coprod 
Q^{(2, 2m)}_{4}, 
\end{align*}
as follows. 
\begin{itemize}
\item $Q^{(2, 2m)}_{0}$ corresponds to 
$p\in Q^{(2, 2m)}$ 
such that $E_p \in \aA_X$ is $Z_{+}$-stable. 
\item $Q^{(2, 2m)}_{1}$ 
corresponds to 
$p\in Q^{(2, 2m)}$ 
such that $E_p \in \aA_X$
 fits into a non-split exact sequence 
\begin{align}\label{E_12}
0 \lr E_1 \lr E_p \lr E_2 \lr 0, 
\end{align}
for $Z_{+}$-stable $E_i \in \aA_X$ 
with $\cl(E_i)=(1, m)$ and $E_1$ is not isomorphic 
to $E_2$. 
\item $Q^{(2, 2m)}_{2}$ corresponds to 
$p\in Q^{(2, 2m)}$ 
such that $E_p \in \aA_X$
is isomorphic to $E_1 \oplus E_2$ 
for $Z_{+}$-stable $E_i \in \aA_X$ 
with $\cl(E_i)=(1, m)$ and $E_1$ is not isomorphic to $E_2$. 
\item $Q^{(2, 2m)}_{3}$ 
corresponds to 
$p\in Q^{(2, 2m)}$ 
such that $E_p \in \aA_X$
fits into a non-split exact sequence 
(\ref{E_12}) such that $E_1 \cong E_2$. 
\item $Q^{(2, 2m)}_{4}$ 
corresponds to 
$p\in Q^{(2, 2m)}$ 
such that $E_p \in \aA_X$
is isomorphic to $E_1^{\oplus 2}$ 
for a $Z_{+}$-stable $E_1 \in \aA_X$. 
\end{itemize}
Then we can write $\delta^{(2, 2m)}(Z_{+}) \in \hH(\aA_X)$ 
as
\begin{align*}
\delta^{(2, 2m)}(Z_{+})=\sum_{i=0}^{4}
\delta_{i}, 
\end{align*}
where $\delta_{i}$ is 
\begin{align}\notag
\delta_i&=\left[\left[\frac{Q^{(2, 2m)}_{i}}{\GL(2, \mathbb{C})}\right]
\to \oO bj(\aA_X)\right] \\
\label{relation}
&=\frac{1}{2}\left[\left[\frac{Q^{(2, 2m)}_{i}}{\mathbb{G}_m^{2}}\right]
\to \oO bj(\aA_X)\right]
-\frac{3}{4}\left[\left[\frac{Q^{(2, 2m)}_{i}}{\mathbb{G}_m}\right]
\to \oO bj(\aA_X)\right].
\end{align}
Here we have used the relation (\ref{relHall}) 
and Example~\ref{exam:F}. 
\begin{lem}
The element $\delta^{(1, m)}(Z_{+})\ast \delta^{(1, m)}(Z_{+})\in 
\hH(\aA_X)$
is written as 
\begin{align}\label{decomp:delta}
\delta^{(1, m)}(Z_{+})\ast \delta^{(1, m)}(Z_{+})
=\sum_{i=1}^{4}\widetilde{\delta}_{i}
\end{align}
where each $\widetilde{\delta}_i$
is as follows. 
\begin{align}\label{delta1}
\widetilde{\delta}_{1}&=
\int_{(p_1, p_2) \in Q^{(1, m)}\times Q^{(1, m)}\setminus D}
\left[ \left[\frac{\mathbb{P}(\Ext^1(E_{p_2}, E_{p_1}))}{\mathbb{G}_m}
\right]
\to \oO bj(\aA_X) \right] d\mu, \\
\label{delta2}
\widetilde{\delta}_{2}&=
\left[\left[\frac{(Q^{(1, m)}\times Q^{(1, m)})\setminus D}
{\mathbb{G}_m^{2}} \right] \to \oO 
bj (\aA_X)\right], \\
\label{delta3}
\widetilde{\delta}_{3}&=
\int_{p\in Q^{(1, m)}}
\left[ \left[\frac{\mathbb{P}(\Ext^1(E_p, E_p))}{\mathbb{A}^1 \times 
\mathbb{G}_m}\right]
\to \oO bj(\aA_X) \right] d\mu, \\
\label{delta4}
\widetilde{\delta}_{4}&=
\left[\left[\frac{Q^{(1, m)}}
{\mathbb{A}^1 \rtimes \mathbb{G}_m^{2}} \right] \to \oO 
bj (\aA_X)\right].
\end{align}
Here $D \subset Q^{(1, m)} \times Q^{(1, m)}$ is 
a diagonal, the algebraic groups in the denominators 
act on the varieties in the numerators trivially. 
The measure 
$\mu$ for the integrations (\ref{delta1}), (\ref{delta3}) 
sends constructible sets on 
$Q^{(1, m)}\times Q^{(1, m)}$ or $Q^{(1, m)}$
to the 
associated elements of the Grothendieck group of 
varieties. 
\end{lem}
\begin{proof}
Recall that $\delta^{(1, m)}(Z_{+})\ast \delta^{(1, m)}(Z_{+})$
is defined by taking the fiber product of the following diagram, 
\begin{align}\label{fiber}
\xymatrix{
& \eE x(\aA_{X}) \ar[d]^{(p_1, p_3)} \\
[Q^{(1, m)}/\mathbb{G}_m] \times [Q^{(1, m)}/\mathbb{G}_m]
\ar[r] & \oO bj(\aA_X) \times \oO bj(\aA_X).}
\end{align}
Here $\mathbb{G}_m$ acts on $Q^{(1, m)}$ trivially. 
Take $\mathbb{C}$-valued points 
$\rho_i \colon \Spec \mathbb{C} \to Q^{(1, m)}$ for $i=1, 2$,
which corresponds to $E_i \in \aA_X$.  
We have the associated elements in the Hall-algebra, 
\begin{align*}
f_i=\left[[\Spec \mathbb{C}/\mathbb{G}_m] \stackrel{\rho_i}{\to}
[Q^{(1, m)}/\mathbb{G}_m] \to \oO bj(\aA_X) \right].
\end{align*}
Then $f_1 \ast f_2$ is as follows, 
\begin{align}\label{f1astf2}
f_{1}\ast f_2 =\left[ \left[ \frac{\Ext^1(E_2, E_1)}{\Hom(E_2, E_1)
\rtimes \mathbb{G}_m^2} \right] 
\to \oO bj(\aA_X)\right].
\end{align}
Here $(t_1, t_2)\in \mathbb{G}_m^2$ acts on $\Ext^1(E_2, E_1)$
and $\Hom(E_2, E_1)$ via multiplying $t_1 t_2^{-1}$, and 
$\Hom(E_2, E_1)$ acts on $\Ext^1(E_2, E_1)$ trivially. 
For $u\in \Ext^2(E_2, E_1)$, 
 the stabilizer group of the $\mathbb{G}_m^2$-action 
 on $\Ext^1(E_2, E_1)$ at $u$ is $\mathbb{G}_m^2$ 
 if $u=0$ and the diagonal subgroup $\mathbb{G}_m \subset \mathbb{G}_m^2$
 if $u\neq 0$. Therefore we have 
 \begin{align}\notag
f_{1}\ast f_2 =&\left[ 
\left[ \frac{\mathbb{P}(\Ext^1(E_2, E_1))}{\Hom(E_2, E_1)
\times \mathbb{G}_m} \right] 
\to \oO bj(\aA_X)\right]+ \\ \label{fiber:decom}
&\qquad \qquad \qquad 
\left[ \left[\frac{\Spec \mathbb{C}}{\Hom(E_2, E_1)
\rtimes \mathbb{G}_m^2} \right] 
\to \oO bj(\aA_X)\right]
\end{align}
Here the algebraic groups in the denominators act trivially on 
the varieties in the numerators. 
Since $E_i \in \aA_X$ are $Z_{+}$-stable, we have 
\begin{align*}
\Hom(E_1, E_2)=\left\{\begin{array}{cc}
\mathbb{A}^1, & \mbox{ if }\rho_1=\rho_2, \\
\Spec \mathbb{C}, & \mbox{ if } \rho_1 \neq \rho_2. 
\end{array} \right. 
\end{align*}
Taking the integration of (\ref{fiber:decom}) 
over points on $(Q^{(1, m)}\times Q^{(1, m)})\setminus D$
and $D\cong Q^{(1, m)}$, 
we obtain the decomposition (\ref{decomp:delta}). 
\end{proof}
\begin{lem}\label{lem:delta0}
The element $\delta_{0} \in \hH(\aA_X)$ is written as (\ref{Ui2}) 
such that $\chi(\delta_{0}, 1)\in \mathbb{Z}$. 
\end{lem}
\begin{proof}
For a point $p \in Q_{0}^{(2, 2m)}$, the 
object $E_p \in \aA_X$
satisfies $\Aut(E_p)=\mathbb{G}_m$ since $E_p$ is $Z_{+}$-stable. 
Hence the diagonal subgroup $\mathbb{G}_m \subset \GL(2, \mathbb{C})$
acts on $Q_{0}^{(2, 2m)}$ trivially, 
and the quotient group $\GL(2, \mathbb{C})/\mathbb{G}_m 
=\PGL(2, \mathbb{C})$
acts freely on $Q_{0}^{(2, 2m)}$. 
Hence $\delta_0$ is written as $[[M/\mathbb{G}_m] \to \oO bj(\aA_X)]$
for an algebraic space $M=Q_{0}^{(2, 2m)}/\PGL(2, \mathbb{C})$, 
and $\mathbb{G}_m$ acts on $M$ trivially.  
Since any algebraic space is written as a disjoint union of 
quasi-projective varieties, $\delta_{0}$ is written as 
(\ref{Ui2}) with each $c_i \in \mathbb{Z}$. Therefore 
$\chi(\delta_{0}, 1)\in \mathbb{Z}$ follows. 
\end{proof}
For $1\le i\le 4$, we set $\epsilon_{i}\in \hH(\aA_X)$
as follows, 
\begin{align*}
\epsilon_i=\delta_i-\frac{1}{2}\widetilde{\delta}_i.
\end{align*}
\begin{lem}\label{lem:e1}
The element $\epsilon_{1}\in \hH(\aA_X)$ is written as 
(\ref{Ui2}) such that $\chi(\epsilon_{1}, 1)\in \mathbb{Z}$. 
\end{lem}
\begin{proof}
For $p\in Q_{1}^{(2, 2m)}$, 
it is easy to see that
the 
object $E_p\in \aA_X$
satisfies $\Aut(E_p)=\mathbb{G}_m$ 
by using the exact sequence (\ref{E_12}). 
Hence $\PGL(2, \mathbb{C})$ 
acts freely on $Q_{1}^{(2, 2m)}$ as in the proof of 
Lemma~\ref{lem:delta0}, and  
the quotient space $Q_{1}^{(2, 2m)}/\PGL(2, \mathbb{C})$ is an algebraic 
space over $\mathbb{C}$. 
Also it is easy to see that the objects $E_i \in \aA_X$
which appear in (\ref{E_12}) are uniquely determined up to 
isomorphisms for a given $p\in Q_{1}^{(2, 2m)}$. 
Hence there is a map of algebraic spaces, 
\begin{align*}
\gamma 
 \colon Q^{(2, 2m)}_{1}/\PGL(2, \mathbb{C}) \to (Q^{(1, m)}\times Q^{(1, m)})
\setminus D,
\end{align*}
such that if $\gamma(p)=(p_1, p_2)$, there is an 
exact sequence in $\aA_X$, 
\begin{align}\label{Ep12}
0 \lr E_{p_1} \lr E_{p} \lr E_{p_2} \lr 0.
\end{align}
By the construction, 
closed points of the fiber of $\gamma$ 
at $(p_1, p_2)$ bijectively correspond to isomorphism classes of objects 
$E_{p}\in \aA_X$ which fit into an exact sequence (\ref{Ep12}), 
which also bijectively correspond to closed points in 
$\mathbb{P}(\Ext^1(E_{p_{2}}, E_{p_{1}}))$. 
Therefore we have 
\begin{align}\notag
\chi(\epsilon_{1}, 1)&=\int_{(p_1, p_2)\in (Q^{(1, m)}\times Q^{(1, m)})
\setminus D}\chi(\mathbb{P}(\Ext^1(E_{p_{2}}, E_{p_{1}}))) d\chi \\ \notag
& \qquad \qquad \quad 
-\frac{1}{2}\int_{(p_1, p_2)\in (Q^{(1, m)}\times Q^{(1, m)})
\setminus D}\chi(\mathbb{P}(\Ext^1(E_{p_{2}}, E_{p_{1}}))) d\chi, \\
\notag
&=\frac{1}{2}\int_{(p_1, p_2)\in (Q^{(1, m)}\times Q^{(1, m)})
\setminus D}\dim \Ext^1(E_{p_{2}}, E_{p_{1}}) d\chi \\
\label{intEx}
&=\int_{(p_1, p_2)\in \Sym^2(Q^{(1, m)})
\setminus D}\dim \Ext^1(E_{p_{2}}, E_{p_{1}}) d\chi \in \mathbb{Z}.
\end{align}
In (\ref{intEx}), we have used the fact that 
$$\dim \Ext^1(E_{p_{2}}, E_{p_{1}})=\dim 
\Ext^1(E_{p_{1}}, E_{p_{2}}), $$
for $(p_1, p_2)\in (Q^{(1, m)}\times Q^{(1, m)})
\setminus D$, which follows from the formula (\ref{RRSe}) 
and 
$$\Hom(E_{p_1}, E_{p_2})=\Hom(E_{p_2}, E_{p_1})=0.$$ 
\end{proof}
\begin{lem}
\label{lem:delta2}
The element $\epsilon_{2} \in \hH(\aA_X)$ is written as (\ref{Ui2}) 
such that $\chi(\epsilon_{2}, 1)=0$.  
\end{lem}
\begin{proof}
Let $T^{G}=\mathbb{G}_m^2 \subset \GL(2, \mathbb{C})$ 
be the subgroup of diagonal matrices, and consider 
the associated $\mathbb{G}_m^2$-action on $Q_{2}^{(2, 2m)}$. 
Since the subgroup $\mathbb{G}_m \subset T^G$
given by (\ref{diag})
acts on $Q_2^{(2, 2m)}$ trivially, 
the quotient group $T^{G}/\mathbb{G}_m \cong \mathbb{G}_m$
acts on $Q_2^{(2, 2m)}$. The set of $T^G/\mathbb{G}_m$-fixed points 
is the image of the map 
\begin{align*}
\iota \colon (Q^{(1, m)} \times Q^{(1, m)})\setminus D 
\to Q_{2}^{(2, 2m)}, 
\end{align*}
defined by 
\begin{align*}
\left((\oO_X \stackrel{s_1}{\twoheadrightarrow}F_1), 
(\oO_X \stackrel{s_2}{\twoheadrightarrow}F_2)\right)
\mapsto (\oO_X^{\oplus 2} \stackrel{(s_1, s_2)}{\twoheadrightarrow}
F_1 \oplus F_2).
\end{align*}
It is easy to see that $\iota$ is an injection, 
and $T^G/\mathbb{G}_m$ acts on 
$Q_{2}^{(2, 2m)}\setminus \Imm \iota$ freely. 
We set $\widetilde{Q}_{2}^{(2, 2m)}$
to be the quotient algebraic space, 
$$\widetilde{Q}_{2}^{(2, 2m)}=
(Q_{2}^{(2, 2m)}\setminus \Imm \iota)/(T^G/\mathbb{G}_m).$$
Noting (\ref{relation}), we obtain that 
\begin{align}\notag
\epsilon_{2}&=\frac{1}{2}\left[\left[\frac{\Imm \iota}{\mathbb{G}_m^2} \right] \to \oO bj (\aA_X)\right]
+\frac{1}{2}\left[\left[\frac{\widetilde{Q}_{2}^{(2, 2m)}}{\mathbb{G}_m}
 \right] \to \oO bj (\aA_X)\right] \\ \notag
&
-\frac{3}{4}\left[\left[\frac{Q_{2}^{(2, 2m)}}{\mathbb{G}_m} \right] \to \oO bj (\aA_X)\right]
-\frac{1}{2}
\left[\left[\frac{(Q^{(1, m)}\times Q^{(1, m)})\setminus D}{\mathbb{G}_m^{2}} 
\right] \to \oO bj(\aA_X)\right] \\
\label{e2written}
&=\frac{1}{2}\left[\left[\frac{\widetilde{Q}_{2}^{(2, 2m)}}{\mathbb{G}_m} \right] \to \oO bj (\aA_X)\right]
-\frac{3}{4}\left[\left[\frac{Q_{2}^{(2, 2m)}}{\mathbb{G}_m} \right] \to \oO bj (\aA_X)\right]. 
\end{align}
Hence $\epsilon_2$ is written as (\ref{Ui2}).
Let us compute the Euler characteristic of $\widetilde{Q}_{2}^{(2, 2m)}$.
For a point $p\in Q^{(2, 2m)}_{2}$ and
the object $E_p\in \aA_X$, take
$(p_1, p_2)\in (Q^{(1, m)}\times Q^{(1, m)})\setminus D$
such that $E_{p}\cong E_{p_1}\oplus E_{p_2}$. 
It is easy to see that 
the pair $(E_1, E_2)$ is uniquely determined 
up to isomorphisms and a permutation. Hence 
$p\mapsto (p_1, p_2)$ 
 defines a well-defined map, 
\begin{align*}
\gamma \colon Q_{2}^{(2, 2m)} \to \Sym^2(Q^{(1, m)})\setminus D. 
\end{align*}
For $(p_1, p_2)\in \Sym^2(Q^{(1, m)})\setminus D$, 
the $\GL(2, \mathbb{C})$-action 
on $Q_{2}^{(2, 2m)}$ induces a map,
$$\GL(2, \mathbb{C}) \twoheadrightarrow \gamma^{-1}(p_1, p_2),$$
which is a $\mathbb{G}_m^2$-bundle over $\gamma^{-1}(p_1, p_2)$. 
Restricting to
 $\gamma^{-1}(p_1, p_2)\setminus \Imm \iota$, 
 we obtain the $\mathbb{G}_m^2$-bundle over
 $\gamma^{-1}(p_1, p_2)\setminus \Imm \iota$, 
\begin{align*}
\GL(2, \mathbb{C})\setminus (T^{G}\cup i(T^{G}))
\twoheadrightarrow \gamma^{-1}(p_1, p_2)\setminus \Imm \iota. 
\end{align*}
Here $i=\left(\begin{array}{cc}0 & 1 \\ 1 & 0 \end{array}\right)
\in \GL(2, \mathbb{C})$. 
Since $\mathbb{G}_m^2$ is a special algebraic group, the
above map is Zariski locally trivial. Hence 
the virtual Poincar\'e polynomial of 
$\gamma^{-1}(p_1, p_2)\setminus \Imm \iota$ is 
\begin{align}\notag
P_t(\gamma^{-1}(p_1, p_2)\setminus \Imm \iota)&=
\frac{P_t(\GL(2, \mathbb{C})\setminus (T^{G}\cup i(T^{G})))}
{P_t(\mathbb{G}_m^2)} \\
\label{virPo}
&=t^4 +t^2 -1.
\end{align}
The free $T^G/\mathbb{G}_m\cong\mathbb{G}_m$-action
 on $Q^{(2, 2m)}\setminus \Imm \iota$
restricts to the free $\mathbb{G}_m$-action
on $\gamma^{-1}(p_1, p_2)\setminus \Imm \iota$.
By (\ref{virPo}), we have 
\begin{align*}
P_t((\gamma^{-1}(p_1, p_2)\setminus \Imm \iota)/\mathbb{G}_m)
&=\frac{t^4+t^2-1}{t^2 -1} \\
&=t^2 +2. 
\end{align*}
By inverting $t=1$, we obtain 
\begin{align}\label{invert}
\chi((\gamma^{-1}(p_1, p_2)\setminus \Imm \iota)/\mathbb{G}_m)=3.
\end{align}
Now the map $\gamma$ descends to a map 
$$\gamma' \colon 
\widetilde{Q}^{(2, 2m)}_{2} \to \Sym^2(Q^{(1, m)})\setminus D,$$
such that the Euler characteristic of each fiber
of $\gamma'$ is $3$
by (\ref{invert}). 
Therefore we obtain 
\begin{align}\notag
\chi(\widetilde{Q}^{(2, 2m)}_{2})
&=3\cdot \chi(\Sym^2(Q^{(1, m)})\setminus D) \\
\label{EuQ}
&=\frac{3}{2}\left(\chi(Q^{(1, m)})^2 -\chi(Q^{(1, m)})\right). 
\end{align}
On the other hand, since the $T^G/\mathbb{G}_m$-fixed 
points in $Q^{(2, 2m)}_2$ coincides with $\Imm \iota$, 
the localization implies 
\begin{align}\label{locali}
\chi(Q^{(2, 2m)}_2)=\chi(Q^{(1, m)})^2 -\chi(Q^{(1, m)}).
\end{align}
By (\ref{e2written}), (\ref{EuQ}) and (\ref{locali}), 
we obtain $\chi(\epsilon_2, 1)=0$. 
\end{proof}
\begin{lem}\label{lem:e3}
The element $\epsilon_{3} \in \hH(\aA_X)$
is written as (\ref{Ui2}) such that 
\begin{align}\label{chie3}
\chi(\epsilon_3, 1)\equiv \frac{m}{2}\chi(Q^{(1, m)}),
 \quad (\mbox{ \rm{mod} }\mathbb{Z} \ ). 
\end{align}
\end{lem}
\begin{proof}
For a point $p\in 
Q_3^{(2, 2m)}$, the object $E_p \in \aA_X$
satisfies 
\begin{align*}
\Aut(E)= \Stab_{p}(\GL(2, \mathbb{C}))\cong 
\mathbb{A}^1 \rtimes \mathbb{G}_m,
\end{align*}
since $E_p$ fits into the exact sequence
(\ref{E_12}) with $E_1 \cong E_2$.
Then for the diagonal matrices
 $T^G \subset \GL(2, \mathbb{C})$,
we have $\Stab_{p}(\GL(2, \mathbb{C}))\cap T^G$
is the subgroup $\mathbb{G}_m \subset T^G$
given by (\ref{diag}). 
Therefore the action of $T^G$ on 
$Q_3^{(2, 2m)}$ descends to the free action 
of $T^G/\mathbb{G}_m \cong \mathbb{G}_m$. 
We set $\widetilde{Q}_3^{(2, 2m)}$
to be the quotient algebraic space, 
\begin{align*}
\widetilde{Q}_3^{(2, 2m)}
=Q_3^{(2, 2m)}/(T^G/\mathbb{G}_m).
\end{align*}
Using (\ref{relation}) and the relation (\ref{relHall}), we have 
\begin{align}\label{form:e3}
\epsilon_{3}=\frac{1}{2}\left[\left[\frac{\widetilde{Q}_3^{(2, 2m)}}
{\mathbb{G}_m}\right]
\to \oO bj(\aA_X) \right]
-\frac{3}{4}\left[\left[\frac{Q_3^{(2, 2m)}}
{\mathbb{G}_m}\right]
\to \oO bj(\aA_X) \right] \\
-\frac{1}{2}\int_{p\in Q^{(1, m)}}\left[ \left[
\frac{\mathbb{P}(\Ext^1(E_p, E_p))}{\mathbb{G}_m} \right] \to \oO bj(\aA_X)\right]. \end{align}
Here the algebraic groups in the denominators 
act on the varieties in the numerators trivially. 
Therefore $\epsilon_3$ is written as (\ref{Ui2}).
Let us calculate the Euler 
characteristic of $\widetilde{Q}_3^{(2, 2m)}$. 
For $p\in Q_3^{(2, 2m)}$, 
let $\gamma(p)\in Q^{(1, m)}$ be the point 
such that $E_p$ fits into the exact sequence (\ref{E_12})
with $E_1\cong E_{\gamma(p)}$. 
It is easy to see that $p \mapsto \gamma(p)$
is a well-defined morphism of 
varieties,  
\begin{align*}
\gamma \colon Q_3^{(2, 2m)} \to Q^{(1, m)}.
\end{align*} 
For $p'\in Q^{(1, m)}$, the fiber of $\gamma$ at $p'$
carries a surjection, 
\begin{align*}
\gamma' \colon \gamma^{-1}(p') 
\twoheadrightarrow \Ext^1(E_{p'}, E_{p'})\setminus \{0\}, 
\end{align*}
which sends a point $p\in \gamma^{-1}(p')$ to the 
extension class of (\ref{E_12}). 
For $u\in \Ext^1(E_{p'}, E_{p'})\setminus \{0\}$, 
we have the surjective morphism, 
\begin{align*}
\gamma''\colon \GL(2, \mathbb{C}) \twoheadrightarrow \gamma^{'-1}(u), 
\end{align*}
induced by the $\GL(2, \mathbb{C})$-action
on $Q^{(2, 2m)}$. 
Each fiber of $\gamma''$ is isomorphic to 
the special algebraic group $\mathbb{A}^1 \rtimes \mathbb{G}_m$, 
hence $\gamma''$ is Zariski locally trivial. 
The free $T^G/\mathbb{G}_m$-action on $Q_{3}^{(2, 2m)}$ restricts
to the free $T^G/\mathbb{G}_m \cong \mathbb{G}_m$-action on 
$\gamma^{'-1}(u)$, and the virtual Poincar\'e polynomial 
of the quotient space is 
\begin{align}\notag
P_t(\gamma^{'-1}(u)/\mathbb{G}_m)&=
\frac{P_t(\GL(2, \mathbb{C}))}{P_t(\mathbb{A}^1 \rtimes \mathbb{G}_m)
P_t(T^G/\mathbb{G}_m)} \\
\label{Ptgamm}
&=t^2 +1.
\end{align}
Now $\gamma'$ descends to a morphism 
$$\gamma^{-1}(E)/\mathbb{G}_m \to \mathbb{P}(\Ext^1(E, E)),$$
such that the Euler characteristic of each fiber is equal to 
$P_t(\gamma^{'-1}(u)/\mathbb{G}_m)|_{t=1}=2$ by (\ref{Ptgamm}). 
Therefore $\chi(\widetilde{Q}_3^{(2, 2m)})$
is 
\begin{align}\label{chiQ}
\chi(\widetilde{Q}_3^{(2, 2m)})=2\int_{p\in Q^{(1, m)}}
\dim \Ext^1(E_p, E_p) \ d\chi. 
\end{align}
Since $\mathbb{G}_m$ acts on $Q_3^{(2, 2m)}$ freely, 
we have $\chi(Q_3^{(2, 2m)})=0$. 
By (\ref{form:e3}) and (\ref{chiQ}), we have 
\begin{align}\label{int:ext}
\chi(\epsilon_3, 1)=\frac{1}{2}\int_{p\in Q^{(1, m)}}
\dim \Ext^1(E_p, E_p) \ d\chi. 
\end{align}
On the other hand, 
the same argument of~\cite[Theorem~4.11]{BBr} shows that 
\begin{align}\label{int:ext2}
\int_{p\in Q^{(1, m)}}
(-1)^{\dim \Ext^1(E_p, E_p)} \ d\chi =(-1)^m \chi(Q^{1, m}). 
\end{align}
By (\ref{int:ext}) and (\ref{int:ext2}), we obtain (\ref{chie3}). 
\end{proof}
\begin{lem}\label{lem:e4}
The element $\epsilon_4 \in \hH(\aA_X)$ 
is written as (\ref{Ui2}) and we have 
\begin{align*}
\chi(\epsilon_4, 1)=-\frac{1}{4}\chi(Q^{(1, m)}).
\end{align*}
\end{lem}
\begin{proof}
By (\ref{relation})
and noting that $F(G, T^G, T^G)=1$, $F(G, T^G, \mathbb{G}_m)=-1$
for $G=\mathbb{A}^1 \rtimes \mathbb{G}_m^2$, 
$T^G=\{0\}\times \mathbb{G}_m^2$ and
$\mathbb{G}_m \subset T^G$ given by (\ref{diag}), 
we have 
\begin{align*}
\epsilon_4 &=
\frac{1}{2}\left[\left[ \frac{Q^{(1, m)}}{\mathbb{G}_m^2}
 \right] \to \oO bj(\aA_X) \right] 
 -\frac{3}{4}\left[\left[ \frac{Q^{(1, m)}}{\mathbb{G}_m}
 \right] \to \oO bj(\aA_X) \right] \\
 &-\frac{1}{2}\left[\left[ \frac{Q^{(1, m)}}{\mathbb{G}_m^2}
 \right] \to \oO bj(\aA_X) \right]
 +\frac{1}{2}\left[\left[ \frac{Q^{(1, m)}}{\mathbb{G}_m}
 \right] \to \oO bj(\aA_X) \right]  \\
 &=-\frac{1}{4}\left[\left[ \frac{Q^{(1, m)}}{\mathbb{G}_m}
 \right] \to \oO bj(\aA_X) \right]. 
\end{align*}
Here the algebraic groups in the denominators act on 
the varieties in the numerators trivially. The above 
formula immediately imply the result. 
\end{proof}
{\bf Proof of Theorem~\ref{thm:intconj}:}
\begin{proof}
By (\ref{OmegaE}), Lemma~\ref{lem:delta0}, Lemma~\ref{lem:e1},
Lemma~\ref{lem:delta2},
Lemma~\ref{lem:e3}
and Lemma~\ref{lem:e4}, 
we obtain 
\begin{align*}
\Omega(2, 2m)&\equiv -\frac{\chi(Q^{(1, m)})}{4}
\{2m-1+(-1)^{m}\} \quad (\mbox{ mod }\mathbb{Z} \ ) \\
&\equiv 0 \quad (\mbox{ mod }\mathbb{Z} \ ). 
\end{align*}
\end{proof}

Institute for the Physics and 
Mathematics of the Universe, University of Tokyo

\textit{E-mail address}:toda-914@pj9.so-net.ne.jp, 
yukinobu.toda@ipmu.jp

\end{document}